\documentclass[11pt]{paper}
\usepackage{amssymb,amsfonts,amsthm,amsmath}
\usepackage{enumerate}
\usepackage{mathtools}
\usepackage{hyperref}
\usepackage{caption}
\usepackage{subcaption}
\usepackage{graphicx}
\usepackage{xcolor}
\hypersetup{colorlinks=true,
linktoc=all,linkcolor=black,
citecolor=black, urlcolor=black}
\mathtoolsset{showonlyrefs}
\usepackage[normalem]{ulem}

\usepackage[a4paper, hmargin=3cm, vmargin={4cm, 4cm}]{geometry}
\usepackage{fancyhdr}
\newtheorem{thm}{Theorem}[section]
\newtheorem{lem}[thm]{Lemma}

\newtheorem{claim}[thm]{Claim}
\theoremstyle{definition}
\newtheorem{defn}[thm]{Definition}
\newtheorem{rem}[thm]{Remark}

\newcommand{\done}{{1\hskip-2.5pt{\rm l}}}
\newcommand{\act}{\mathcal{E}}
\newcommand{\work}{\mathcal{W}}
\newcommand{\flux}{\mathcal{F}}
\newcommand{\Var}{{\sf Var}}
\newcommand{\Cov}{{\sf Cov}}
\newcommand{\pv}{{\sf p.v.}}
\newcommand{\C}{{\mathbb C}}
\newcommand{\D}{{\mathbb D}}
\newcommand{\T}{{\mathbb T}}
\newcommand{\R}{{\mathbb R}}

\newcommand{\G}{\mathcal{G}}

\newcommand{\E}{\mathbb{E}}
\renewcommand{\P}{\mathbb{P}}
\newcommand{\La}{\Lambda}
\newcommand{\la}{\lambda}
\newcommand{\bP}{\mathbb{P}}
\newcommand{\calC}{\mathcal{C}}
\newcommand{\calL}{\mathcal{L}}
\newcommand{\calS}{\mathcal{S}}
\newcommand{\calG}{\mathcal{G}}
\newcommand{\diff}{{\rm d}}
\newcommand{\say}[1]{``#1''}
\newcommand{\ed}{\overset{{\sf def}}{=}}
\renewcommand{\Re}{\operatorname{Re}}

\renewcommand{\epsilon}{\varepsilon}

\numberwithin{equation}{section}

\begin{document}

\title{The random Weierstrass zeta function II. \\
Fluctuations of the electric flux through rectifiable curves}
\author{Mikhail Sodin, Aron Wennman and Oren Yakir}{}

\maketitle

\begin{abstract}
Consider a random planar point process whose
law is invariant under planar isometries. We think of the process as a random
distribution of point charges and consider the electric field generated by the
charge distribution.
In Part I of this work, we found a condition on the spectral side which 
characterizes when the field itself is invariant 
with a well-defined second-order structure.
Here, we fix a process with an invariant field,
and study the fluctuations of the flux through 
large arcs and curves in the plane. Under suitable
conditions on the process and on the curve, denoted $\Gamma$, we show that 
the asymptotic variance of the flux through $R\,\Gamma$ grows like $R$ times
the signed length of $\Gamma$. 
As a corollary, we find that the charge fluctuations in a dilated Jordan domain 
is asymptotic with the perimeter, provided only that the boundary
is rectifiable.

The proof is based on the asymptotic analysis of a closely related 
quantity (the \emph{complex electric action} of the field along a curve). 
A decisive role in the analysis is played by a signed version of 
the classical Ahlfors regularity
condition.
\end{abstract}

\section{Introduction}
\subsection{Electric fields and charge fluctuations}
Denote by $\Lambda$ a stationary random point process in the 
complex plane $\C$ with intensity $c_\La$ 
(i.e.\ the mean number of points of $\La$ per unit area). 
If we think of $\La$ as a random distribution of identical point charges, 
each sample generates an electric field which we identify
with a solution $V_\Lambda$ to the equation
\begin{equation}
\label{eq:V-eq}
\bar\partial \, V_\Lambda= \pi \sum_{\lambda\in\Lambda}\delta_\lambda
-\pi c_\La\hspace{1pt} m  ,
\end{equation}
where $\bar\partial=\frac12(\partial_x+{\rm i} \, \partial_y)$ 
and where $m$ is the Lebesgue measure.
In \cite{SodinWennmanYakirPreprint1} (henceforth referred to as Part {\rm I}), 
we found that, for processes with a spectral measure $\rho_\La$,
the spectral condition
\begin{equation}
\label{eq:spectral-cond}
\int_{|\xi|\le 1}\frac{\diff\rho_\La(\xi)}{|\xi|^2}<\infty
\end{equation}
characterizes those processes $\Lambda$ for which 
there exists a \emph{stationary} electric field $V_\La$ with a 
well-defined second-order (covariance) structure. 
When it exists, the field is given by 
\begin{equation}
\label{eq:V-def}
V_\La(z)=\lim_{R\to\infty}\sum_{|\la|<R}\frac{1}{z-\la}-\pi c_\La\hspace{1pt}\overline{z},
\end{equation}
and furthermore, it is essentially unique. In addition, 
the covariance structure of $V_\La$ is conveniently 
expressed in terms of the covariance structure of $\Lambda$. 
The field $V_\Lambda$ can be seen as a random analogue of
the Weierstrass zeta function from the theory of elliptic functions.

Denote by $n_\La(\calG)=\#(\La\cap \calG)$ the 
counting measure (or \say{charge}) 
of a Borel set $\calG$.
The asymptotic variance $\Var[n_\La(R\calG)]$
of the charge fluctuations as $R\to\infty$ 
is of central interest in statistical physics.
For the Poisson process, the size of the charge 
fluctuations in a disk is proportional to the
area, while for more negatively correlated point processes 
the fluctuations tend to be suppressed
and grow like $o(R^2)$. Following Torquato-Stillinger, such 
processes are called \emph{hyperuniform} or 
\emph{super-homogeneous}; see \cite{TorquatoStillinger}. 
Classical examples of hyperuniform point processes
include the Ginibre ensemble \cite{JancoviciCoulomb} and the zero set 
of the Gaussian Entire Function (GEF) \cite{BuckleySodinJSP}. 
For both these examples,
the variance of $n_\La(R\D)$ grows like the perimeter $R$.
It is curious to ask {\em to what extent the geometry of the disk
is important, and in particular, what role is played by boundary 
regularity}.

Since the divergence theorem expresses the charge in a domain as 
the electric flux through its boundary, it seems natural 
to ask for the behavior of the fluctuations of the flux 
of the field $V_\La$ through more general rectifiable curves, which need not be 
closed nor simple. Specifically, given 
a rectifiable curve $\Gamma$ with unit normal $N$, we 
would like to describe the asymptotic variance of
\[
\flux_\La(\Gamma)=\int_{R\Gamma}V_\La(z)\cdot N\,|\diff z|
\]
for large $R$, provided that $\La$ satisfies \eqref{eq:spectral-cond}. 
To isolate the effect of the geometry of $\Gamma$, 
we will make the assumption that, in addition to being stationary, the 
law of the point process $\Lambda$ is invariant
under rotations. We refer to such a process 
simply as \say{invariant}. 

\subsection{The electric action}
\label{s:action}
For an invariant point process $\La$ subject to 
the spectral condition \eqref{eq:spectral-cond}, 
we introduce the \emph{electric action}
\begin{equation}
\label{eq:flux-formula}
\act_\Lambda(\Gamma)=
\int_{\Gamma}V_\Lambda(z) \, \diff z
\end{equation}
where ${\diff z}$ denotes the usual holomorphic $1$-differential
and where $\Gamma$ is an arbitrary rectifiable curve in the plane.
When $\Gamma$ bounds a Jordan domain $\calG$, 
the action coincides with the flux, which in
turn equals
$2{\rm i}$ times $n_\La(\calG)-c_\La m(\calG)$.
In general, if $T$ and $N$ denote the unit tangent and normal vectors
to $\Gamma$, respectively, we obtain a decomposition
\[
\act_\Lambda(\Gamma)=\int_{\Gamma} 
\Big(V_\Lambda \cdot T + {\rm i} \, V_\Lambda\cdot N\Big) \, |\diff z|
\ed \work_\Lambda(\Gamma)+{\rm i} \, \flux_\Lambda(\Gamma)
\] 
of $\act_\Lambda(\Gamma)$ 
in terms of the flux $\flux_\La(\Gamma)$ and the 
work $\work_\Lambda(\Gamma)$ of the field along 
$\Gamma$. While we are mainly interested in flux, the 
action appears to be more natural from an analytical 
point of view.
Note that the work $\work_\La(\Gamma)$ vanishes when the curve is closed.
It appears that in the general situation, the fluctuations of $\work_\La(R\Gamma)$
are negligible compared with those of the flux $\flux_\La(R\Gamma)$ 
(see Theorem~\ref{thm:flux} below).

\subsection{Main results}
We assume throughout that $\La$ has a finite second moment,
that is, that $\E\left[n_\La(B)^2\right]<\infty$ for any bounded Borel set $B$.
Under this assumption, there exists a measure $\kappa_\La$  
(\emph{the reduced covariance measure}) with the property that
\[
\Cov  \left[n_\La(\varphi),n_\La(\psi)\right]=
\iint_{\C\times\C}\varphi(z)\, \overline{\psi(z')}\, \diff\kappa_\La(z-z^\prime)
\, \diff m(z)
\]
for any test functions $\varphi,\psi\in C^\infty_0(\C)$, 
where $n_\La(\varphi)$ is the standard
linear statistic given by
$n_\La(\varphi)=\sum_{\la\in\La}\varphi(\la)$.
The spectral measure $\rho_\La$ is the Fourier transform of $\kappa_\La$ 
(understood in the sense of distributions).

We will also assume that \emph{the reduced truncated two-point measure} 
$\tau_\La\ed \kappa_\La-c_\La\delta_0$ has  
a density $k_\La$ with respect to planar Lebesgue 
measure (the \emph{truncated two-point function}). 
Here, $c_\La$ is the intensity of $\La$, and $\delta_0$ is the unit
point mass at the origin.

We recall also the standing assumption that the stationary point process 
$\La$ is \emph{invariant}, i.e., that the law of $\La$ is invariant under all 
planar isometries. Under this assumption, the two-point function is 
automatically radially symmetric, so that
\begin{equation}
\label{eq:radial-two-pt}
\kappa_\La(z)=k_\La(|z|)\, \diff m(z)+c_\La\delta_0(z).
\end{equation}
Our results will require that $(1+t^2)k_\La(t)\in L^1(\R_{\ge 0},{\rm d}t)$.
Then, by [Part {\rm I}, Remark~5.3], the spectral 
measure $\rho_\La$ has a radial density $h_\La$ which is $C^1$-smooth. 
Furthermore, if $\kappa_\La(\C)=0$ (i.e., $\displaystyle \int_{\C} k_\Lambda (|z|) \, {\rm d} m(z) = - c_\La$), 
we have the identity
\begin{equation}
\label{eq:h-k}
\int_0^\infty k_\La(t) \, t^2 \, {\rm d}t=
-\frac{1}{4\pi^2}\int_{0}^\infty\frac{h_\La(\tau)}{\tau^2} \, {\rm d}\tau,
\end{equation}
see Remark~\ref{rem:formula-CLambda} below.
Hence, the above moment assumption implies that $\rho_\La$ satisfies 
the spectral condition \eqref{eq:spectral-cond}.

By a \emph{curve} we mean a continuous map $\Gamma:I\to \C$ of an interval $I$ into the plane.
We identify two maps if they differ by pre-composition with an order-preserving
homeomorphism of two intervals. 
We stress that $\Gamma$ need not be injective; 
if this is required we call the curve simple.
The curve $\Gamma$ is said to be {\em rectifiable}
if it has finite \emph{length}
\[
|\Gamma|=\sup_{t_0<t_1<\ldots<t_n}\sum_{j=1}^{n}\big|\Gamma(t_{j})-\Gamma(t_{j-1})\big|
=\int_I|\gamma'(t)|\,\diff t<\infty,
\]
where the supremum is taken over all partitions $a\le t_0<t_1<\ldots<t_n\le b$ of $I=[a,b]$,
and where, without loss of generality, it is tacitly assumed
that $\gamma$ is a Lipschitz parametrization of $\Gamma$, so that $\gamma'$ exists a.e..
When no confusion should arise, we will abuse terminology somewhat and 
identify a curve with its image. 
When referring to a particular parametrization,
we will use lowercase Greek letters, and write e.g.\ $\gamma:I\to\Gamma$.

The following regularity notion plays a key role in our analysis.

\begin{defn}[Weak Ahlfors regularity]
\label{def:weak-reg}
We say that a rectifiable curve $\Gamma$ is 
\emph{weakly Ahlfors regular} if there exists 
a constant $C_\Gamma<\infty$ such that
\begin{equation}
\label{eq:def-weak-reg}
\left|\int_{\Gamma\cap D}\diff z\right|\le C_\Gamma |\partial D|,
\end{equation}
for any Euclidean disk $D\subset \C$.
\end{defn}

The terminology is borrowed from the classical Ahlfors regularity
condition, 
which asks that
\[
\int_{\Gamma\cap D}|\diff z|\le C_\Gamma|\partial D|.
\] 
While rectifiable Jordan curves 
may fail to be Ahlfors regular,
it turns out that they are always weakly Ahlfors; 
see Lemma~\ref{lem:weak-Ahlfors}. Informally speaking, the weak Ahlfors condition
is meant to prevent excessive spiralling near the end points of $\Gamma$, which appears 
to be the main obstruction to linear growth of charge fluctuations.

\begin{figure}[t!]
\centering
\begin{subfigure}[t]{.5\textwidth}
\centering
\includegraphics[width=\linewidth]{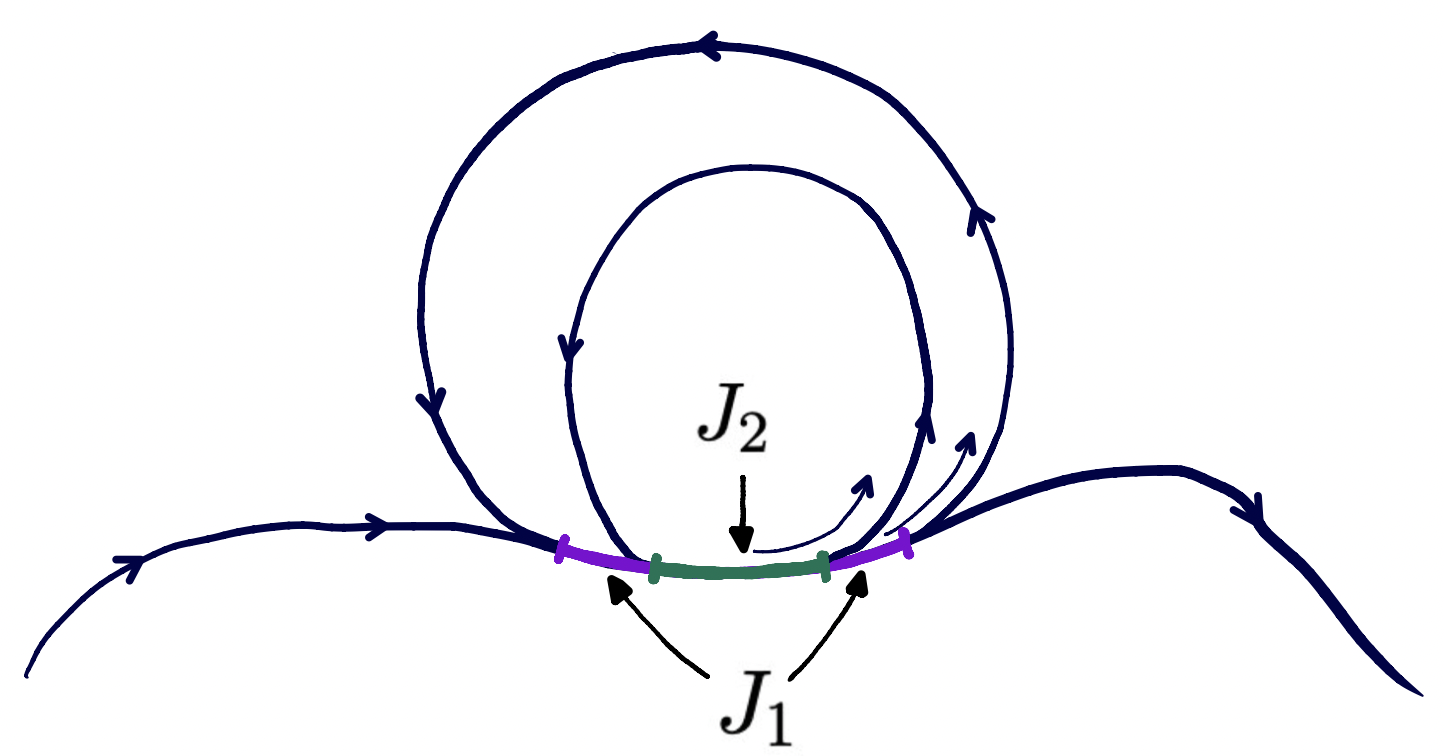}
\end{subfigure}
\hspace{10pt}
\begin{subfigure}[t]{.45\textwidth}
\includegraphics[width=\linewidth]{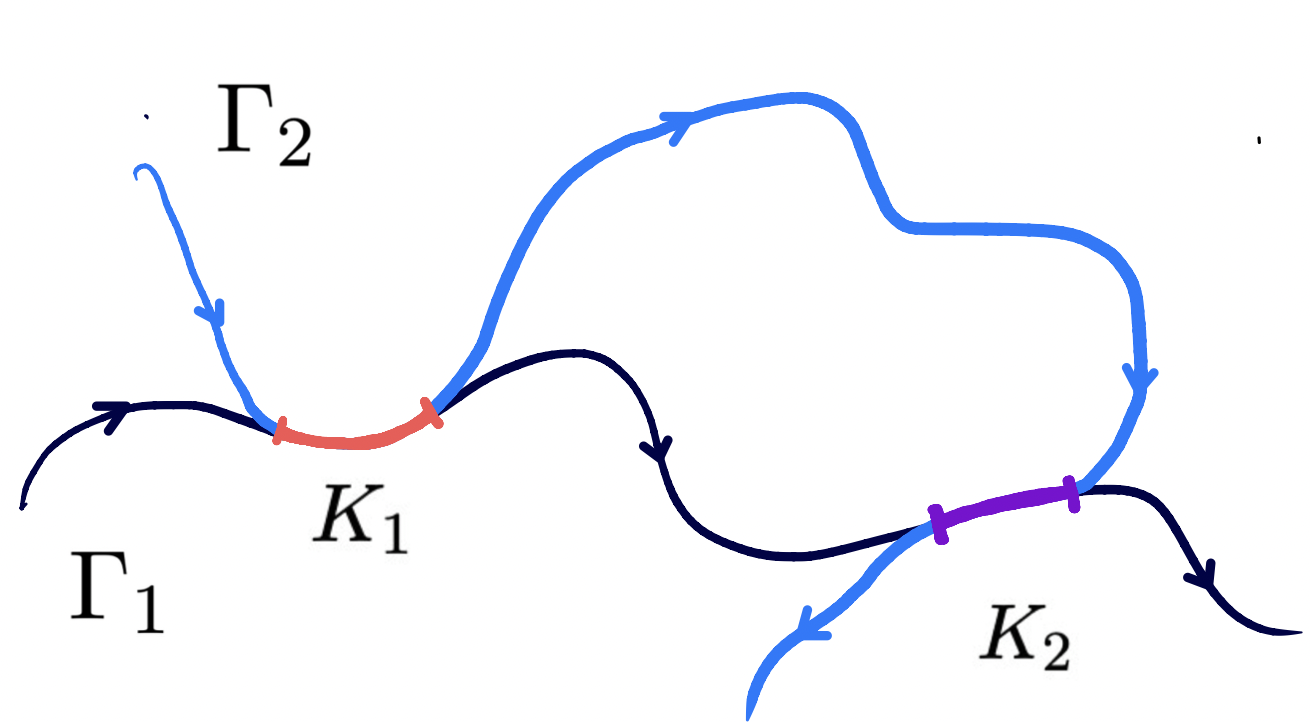}
\end{subfigure}
\vspace{20pt}
\caption{{\sf Left:} The purple arcs $J_1\subset\Gamma$ are traversed 
twice and the green arc $J_2$ 
is traversed three times; these contribute $4\mathcal{H}^1(J_1)$ 
and $9\mathcal{H}^1(J_2)$, respectively, to the signed length $\mathcal{L}(\Gamma,\Gamma)$. 
{\sf Right:} For the signed length $\mathcal{L}(\Gamma_1,\Gamma_2)$, 
the arc $K_1$ (orange) is counted with a positive sign,
while $K_2$ (purple) is counted with negative sign.}
\label{fig:signedLength}
\end{figure}

\begin{defn}[Signed length]
\label{def:signed-length}
The \emph{signed length} of the intersection 
of two rectifiable curves $\Gamma_1$ and $\Gamma_2$ is given by
\begin{equation}
\label{eq:signed-length}
\calL(\Gamma_1,\Gamma_2)
=\int_{\C}\bigg(\sum_{s\in \gamma_1^{-1}(z),\; t\in\gamma_2^{-1}(z)}
\gamma_1^\prime(s) \cdot \gamma_2^\prime(t)\bigg)\,\diff\mathcal{H}^1(z),
\end{equation}
where $\gamma_1$ and $\gamma_2$ denote the 
arc-length parametrizations of the two curves, i.e.\ parametrizations with
$|\gamma'_j|=1$ a.e., and where $\mathcal{H}^1$ is one-dimensional 
Hausdorff measure on $\C$.
\end{defn}
The quantity~\eqref{eq:signed-length}, which was introduced in 
\cite{BuckleySodinJSP}, measures the length of
the intersection $\Gamma_1\cap\Gamma_2$ taking orientation and multiplicity into account; 
see Figure~\ref{fig:signedLength}.
The signed length 
$\mathcal{L}(\Gamma_1,\Gamma_2)$ is finite for any two weakly 
Ahlfors regular curves;
see Remark~\ref{rem:finite-signed-length}. 

\begin{thm}
\label{thm:main}

Denote by $\La$ an invariant
point process and let $\mathcal{E}_{\La}(\Gamma) = \int_{\Gamma} V_\Lambda(z) \, \diff z$.
Assume that the two-point function $k_\La$ 
of $\La$
satisfies $(1+t^2)k_\La(t)\in L^1(\R_{\ge 0},\diff t)$ together with the zeroth
moment condition 
\begin{equation}
\label{eq:zeroth-moment}
\int_0^\infty k_\La(t)t\,{\rm d}t=-c_\La.
\end{equation}
Then, for any weakly Ahlfors regular 
rectifiable curves $\Gamma_1$ and $\Gamma_2$ 
we have that
\begin{equation}
\label{eq:asymp-main-thm}
\Cov  
\big[\act_\Lambda(R\, \Gamma_1) \, ,
\act_\Lambda(R \, \Gamma_2)\big]
=R\left(C_\La + o(1)\right)\calL(\Gamma_1,\Gamma_2)
\end{equation}
as $R\to\infty$, where $\displaystyle{C_\Lambda=-8\pi^{2}\int_0^\infty k_\La(t)\,t^2 \, \diff t}$. 
\end{thm}

Note that the identity \eqref{eq:h-k} makes it clear that $C_\La$ is positive.
Since Jordan curves are always weakly Ahlfors regular (see Lemma~\ref{lem:weak-Ahlfors}), 
we get the following result.

\begin{thm}
\label{thm:main2}
Assume that $\Lambda$ satisfies the conditions of 
Theorem \ref{thm:main}. Then for any Jordan domain $\calG$
with rectifiable boundary, we have the asymptotics
\begin{equation}
\label{eq:asympPositive}
\Var \left[n_\Lambda(R\calG)\right]=
\tfrac{1}{4}R\left(C_\La+o(1)\right)|\partial\calG|
\end{equation}
as $R\to\infty$.
\end{thm}
It should be mentioned that the fact that the zeroth-moment 
condition~\eqref{eq:zeroth-moment} yields the suppressed charge fluctuations~\eqref{eq:asympPositive} 
for domains $\calG$ with smooth boundaries was known a long time ago, see~\cite[Section~1.6]{MartinSumRules}	
\begin{rem}
\label{rem:reg-cond-relax}
Our proof of Theorem~\ref{thm:main} yields that the weak Ahlfors condition
can be relaxed to the maximal function criterion 
for the pair of rectifiable curves $(\Gamma_1,\Gamma_2)$:
\[
\sup_{0<\epsilon\le 1}\epsilon^{-1}\int_0^\epsilon \frac{1}{t}
\left|\int_{\Gamma_1\cap \D(\zeta,t)}{\diff z}\right|\diff t 
\in L^1(\Gamma_2,|\diff\zeta|)
\]
without essential modifications. We have chosen to work 
with the stronger condition in Definition~\ref{def:weak-reg}
to make the proofs as transparent as possible, but by doing 
so we miss out on a few examples where the 
maximal function condition would be needed. 
For instance, this seems to be the case for the \say{nested squares} curve
$\Gamma_1=\Gamma_2=\bigcup_{k\ge 0}2^{-k}\partial Q$ where $Q=[0,1]\times [0,1]$.
\end{rem}

There ought to exist counterexamples
to the asymptotics of Theorem~\ref{thm:main}
if we do not impose anything like weak regularity.
At least this is the case for the Ginibre ensemble, 
for which we show that for any 
$\epsilon>0$, there exists a rectifiable Jordan arc 
$\Gamma_\epsilon$ for which
\begin{equation}
\label{eq:counterEx}
\Var \big[\act_\Lambda(R\Gamma_\varepsilon)\big]\gtrsim R^{2-\epsilon}|\Gamma|
\end{equation}
holds as $R\to\infty$; see \S\ref{s:Gin-counter}--\S\ref{s:counter}. 

Our last result concerns the asymptotics of the work $\work_\La(\Gamma)$, i.e.
the real part of the complex action $\act_\La(\Gamma)$ (defined in \S\ref{s:action}).

\begin{thm}
\label{thm:flux}
Assume that the two-point function $k_\La$ of the invariant point
process $\Lambda$ satisfies $(1+t^3)k_\La(t)\in L^1(\R_{\ge 0}, \diff t)$ 
along with the zeroth moment condition \eqref{eq:zeroth-moment}, and that
$\Gamma$ is a weakly Ahlfors regular rectifiable curve with distinct start and end points. 
Then, as $R\to\infty$,
\[
\Var \big[\work_\Lambda(R\Gamma)\big]
=\left(D_\Lambda+o(1)\right) \log R
\]
where $\displaystyle{D_\Lambda=2\pi^2 \int_0^\infty t^3 k_\La(t)\diff t}$. 
As a consequence, we have that, as $R\to\infty$,
\[
\Var \big[\act_\La(R\Gamma)\big]
=\Var \big[\flux_\Lambda(R\Gamma)\big]  + O(\log R) \, . 
\]
\end{thm}

In other words, under these assumptions the quantity $\act_\Lambda(R\Gamma)$ really 
measures the asymptotic electric flux through $R\Gamma$.

\begin{rem}
Denoting by $(\Omega,\mathcal{F},\P)$ the probability space on which the stationary 
point process $\La$ is defined, let $\mathcal{F}_{\sf inv} \subset\mathcal{F}$ 
be the sigma-algebra of translation invariant events. 
Throughout this paper, we make the simplifying assumption
that the spectral measure $\rho_\La$ does not have an atom at the origin. As outlined in
[Part {\rm I}, \S2.3], this is the same as saying that the \emph{conditional intensity} of $\La$,
defined by
\[
\mathfrak{c}_\La = \pi^{-1} \E\big[n_\Lambda(\D) \mid \mathcal{F}_{{\sf inv}}\big] \,,
\]
is non-random and equals $c_\La$ almost surely. While this assumption 
simplifies the formulation of our results, 
our methods can be adapted to deal also with the case when $\rho_\La(\{0\})>0$. 
We will not dwell on the details, but mention only that in this setting,
the stationary vector field should be defined as in Part {\rm I} by
\[
V_{\La}(z) = \lim_{R\to\infty}\sum_{|\la|<R}\frac{1}{z-\la}
-\pi \mathfrak{c}_\La \overline{z} \,,
\]
and, in Theorem~\ref{thm:main2}, one should replace the variance $\Var[n_\La(R\,\calG)]$ 
with the re-centered variance
$\Var\big[n_\La(\calG)-\mathfrak{c}_\La m(\calG)\big]$.
Since the spectral measure $\rho_{V_\La}$ and
two-point function $K_{V_\La}$ do not depend on the size of the atom $\rho_\La(\{0\})$
([Part {\rm I}, Theorem 5.8]), 
only very minor modifications will be needed.
\end{rem}

It is natural to ask about asymptotic
normality of the renormalized electric flux $R^{-1/2}\flux_\La(R\Gamma)$
(or, equivalently, of the renormalized electric action).
For the zeroes of GEFs and $C^1$-smooth curves, this was proven in \cite{BuckleySodinJSP} 
using the method of moments. An alternative is to use
the clustering property of $k$-point functions
(cf. Malyshev~\cite{Malyshev}, Martin-Yalcin~\cite{MartinYalcin}, 
Nazarov-Sodin~\cite[Theorem~1.5]{NS-CMP}), which is easy to verify for the Ginibre
ensemble using its determinantal structure, 
and which was proven in~\cite{NS-CMP} for the zero
set of GEFs. We will not pursue the details here.

\subsection{Related work}
The study of charge fluctuations of stationary 
point processes is a classical topic in mathematical 
physics; see \cite{GhoshLeb} for a recent survey. 
Important early contributions were the works of 
Martin-Yalcin \cite{MartinYalcin}, Lebowitz \cite{LebFluct}, 
and Jancovici-Lebowitz-Manificat \cite{JLM} (see also 
Martin's survey \cite{MartinSumRules}).
Recently there has been a resurgence of 
interest, due partly to the role of hyperuniform systems 
in material science (see e.g.\ \cite{Torquato}). 
This has highlighted the relationship between 
charge fluctuations and properties of the spectral 
measure; see the work~\cite{AdhikariGhoshLebowitz} of
Adhikari-Ghosh-Lebowitz and 
references therein for recent mathematical 
developments.

If we do not impose any smoothness condition on the spectral measure, 
there exist examples where the asymptotics~\eqref{eq:asympPositive} cease to hold. 
To see this, simply consider the stationary point process obtained by 
randomly shifting the lattice points $\mathbb{Z}^2\subset \R^2 = \C$ 
by a random variable uniformly distributed on $[0,1]^2$. 
In this example, the spectral measure consists of unit point masses 
on $\mathbb{Z}^2\setminus \{0\}$. Denoting by $\mathbb{B}^p\subset \R^2$ 
the unit ball in the $\ell^p$-metric, Kim and Torquato~\cite{KimTorquato} 
showed that $\Var \big[n_\Lambda(R \, \mathbb{B}^p) \big]$ grows 
like $R^{\sigma(p)}$, where $\sigma(p)$ varies continuously between 
$1$ and $2$ as $p$ ranges from $2$ to infinity. It is also worth 
mentioning that examples of this sort persists when considering 
independent Gaussian perturbations of the ``randomly shifted" lattice, 
where the spectral measure is a mixture of an absolutely continuous 
part and a singular part, see Yakir~\cite{Yakir}. Another related 
work is the recent study by Bj\"orklund and Hartnick~\cite{BjorklundHartnick}, 
which investigated the hyperuniformity of various point processes 
which are models for quasicrystals. In this case the asymptotic of 
the number variance may depend on fine arithmetic properties of the quasicrystal.

For the zeros of the Gaussian Entire Function $F(z)$, 
the flux $\flux_\La(\Gamma)$ corresponds to
the change in argument of $F(z)$ along $R\Gamma$. This was introduced by Buckley-Sodin in 
\cite{BuckleySodinJSP}, though the idea was present already in Lebowitz work \cite{LebFluct}
on charge fluctuations in Coulomb systems. 
Buckley and Sodin obtained the large $R$-asymptotics of 
$\Cov\left[\flux_\La(R\Gamma_1),\flux_\La(R\Gamma_2)\right]$ 
for piecewise $C^1$-regular curves $\Gamma_1$ and $\Gamma_2$ with a proof
which used the properties of that ensemble. It is not clear whether the proof given there
persists for weakly Ahlfors regular curves. We mention that the signed length recently appeared in the
work of Notarnicola-Peccati-Vidotto on the length of nodal
lines in Berry's random planar wave model~\cite{NPV}.

For determinental point processes (including the infinite Ginibre ensemble) Lin~\cite{Lin} 
very recently obtained a stronger version of Theorem~\ref{thm:main2}, which is 
valid for a more general class of domains $\calG$ having a bounded perimeter. 
He also showed that for a very large class of domains $\calG$, $\Var\big[n_\La(R\calG)\big]$ 
is comparable to $R^{\alpha(\calG)}$, where $\alpha(\calG)$ is the 
Minkowski dimension of $\partial \cal G$. Lin's approach is based 
on the study of asymptotics of the functional 
\begin{align*}
\iint_{\R^2\times \R^2} \big(\done_{\calG}(x) - \done_{\calG}(y)\big)^2& 
\rho\big(R|x-y|\big) \, {\rm d}m(x)\, {\rm d}m(y) \\ & 
=\iint_{\calG \times \calG^c} \rho\big(R|x-y|\big) \, {\rm d}m(x)\, {\rm d}m(y)
\end{align*}
as $R\to \infty$, where $\rho$ is a non-negative integrable function 
which is sufficiently fast decaying at infinity (cf. D\'avila~\cite{Davila}). 
In our case $\rho=-k_\La$, the (truncated) two-point function of $\La$ 
taken with negative sign. Non-negativity of $\rho$ seems to be an obstacle 
for using this technique for non-determinental point processes. 
Nevertheless, Lin mentions~\cite[Remark~1.2.3]{Lin} that his techniques 
can be also applied to zeros of Gaussian Entire Functions.

There is also a resemblance between the topic of this paper
and the study of \say{irregularities of distribution} e.g. in the work of 
Montgomery \cite[Ch. 6]{Montgomery}.

\subsection{Outline of the paper}
The starting point for the asymptotic analysis of 
$\Var\big[\act_\Lambda(R\Gamma)\big]$ is the identity
\begin{equation}
\label{eq:var-formula}
\Var\big[\act_\Lambda(R\Gamma)\big]
={\sf p.v.}\iint_{R\Gamma \times R\Gamma} K(|x-y|) \, \diff x \diff\bar{y},
\end{equation}
where $K$ is the (singular) covariance 
kernel of $V_\Lambda$ (see \S\ref{s:cov-flux})
and where the principal value integral is understood in the sense
of the limit 
\[
\lim_{\epsilon\to 0}\iint_{R\Gamma\times R\Gamma}
\done_{\{|x-y|>\epsilon\}}
K(|x-y|)\,\diff x\diff \bar{y}.
\]

The article is organized as follows. In \S\ref{s:prel} 
we recall various preliminaries, mainly on the second-order 
structure of stationary processes. Section~\ref{s:geom} is devoted
to geometric observations concerning the weak-Ahlfors condition. 
In \S\ref{s:cov-flux} we establish the formula \eqref{eq:var-formula} for the variance
and recall a convenient representation of the kernel $K$ from 
Part {\rm I}. The proof of the main results 
are given in \S\ref{s:pf-main}, and the existence
of rectifiable but not weakly Ahlfors regular Jordan 
arcs with large charge fluctuations is discussed in 
\S\ref{s:counter-ex} (this is made rigorous for the infinite Ginibre ensemble).

\section{Preliminaries}
\label{s:prel}
\subsection{Notation and conventions}
\label{s:not}

We will frequently use the following notation.
\begin{itemize}
\setlength\itemsep{0.25em}
\item $\C$, $\R$, $\R_{\ge 0}$; the complex plane, the real 
line and the half line $\R_{\ge 0}=\{x\in\R:x\ge 0\}$

\item $\D(x,R)$; the disk $\{z\in\C:|z-x|<R\}$. We let $\D(0,1)=\D$

\item $\partial=\partial_z$ and $\bar\partial
=\partial_{\bar z}$; the Wirtinger derivatives
\[
\partial=\frac12\left(\frac{\partial}{\partial x}
-{\rm i}\frac{\partial}{\partial y}\right),\qquad
\bar\partial=\frac12\left(\frac{\partial}{\partial x}
+{\rm i}\frac{\partial}{\partial y}\right)
\]

\item $m$; the Lebesgue measure on $\C$

\item $\widehat{f}$; the Fourier transform, with the normalization
\[
\widehat{f}(\xi)=\int_{\C}\,e^{-2\pi {\rm i}x\cdot \xi}f(x)\,{\diff}m(x)
\]

\item $\mathfrak D$, $\mathcal S$; the class of 
compactly supported $C^\infty$-smooth functions
and the class of Schwartz functions, respectively

\item $\E$, $\Cov$, $\Var$; the expectation, 
covariance and variance with respect to the probability space
$(\Omega,\mathcal{F},\bP)$ on which $\La$ is defined

\item $K$; the two-point function for the stationary field 
$V_\La$; see Lemma~\ref{lem:var-V}.

\item $\rho_\La$, $\kappa_\La$, $\tau_\La$; spectral measure, and the reduced and reduced 
truncated covariance measures for $\La$; see \S\ref{ss:cov-prel}.

\item $c_\La$; the intensity of $\La$

\item $\tau(x)=\tau_\Gamma(x)$; the \say{net tangent} 
$\tau(x)\ed\sum_{t\in\gamma^{-1}(x)}\gamma'(t)$
for an oriented rectifiable curve $\Gamma$; 
see \eqref{eq:net-tangent}

\item $n_\La$; the random counting measure  
$n_\Lambda=\sum_{\lambda\in\Lambda}\delta_\lambda$

\item $\mu(f)=(f,\mu)$; the distributional action of the 
measure $\mu$; i.e.\ $\mu(f)=\int f\,{\diff \mu}$

\item $\nu_\Gamma$; the current which acts on one-forms $\omega$ by
$\nu_\Gamma(\omega)= \int_{\Gamma}{\diff}\omega$. 
By a slight abuse of notation, we will write $\nu_\Gamma(f)$
to mean $\nu_\Gamma(f \, \diff z)$. We will identify $\nu_\Gamma$
with a complex-valued finite measure by setting 
$\nu_\Gamma(B)=\int_{\Gamma\cap B}\diff z$
for Borel sets $B\subset\C$.

\item $\phi_\epsilon$; the Gaussian $\phi_\epsilon(z)
=\frac{1}{\pi\epsilon^2}e^{-|z|^2/\epsilon^2}$.
\end{itemize}
We use the standard $O$-notation and the notation 
$\lesssim$ with interchangeable meaning. If a limiting 
procedure involves an auxiliary parameter $a$, we write e.g.\
$f_a(x)=O_a(g(x))$ to indicate the dependence of the 
implicit constant on the parameter.

\subsection{The covariance structure of 
\texorpdfstring{$\La$}{Lambda}}
\label{ss:cov-prel}
For a discussion of the spectral measure and various 
covariance measures of the point process $\La$, we refer to 
\S2 of Part {\rm I}. For the reader's 
convenience, we recall the most central notions here.

We assume throughout that the point process has a finite 
second moment, i.e.\ that $\E \big[n_\La(B)^2\big]<\infty$ holds 
for any bounded Borel set $B$. Under this assumption there 
exists a measure $\kappa_\La$ such that
\[
\Cov  \big[n_\La(\varphi),n_\La(\psi)\big]
=\iint_{\C\times\C}\varphi(z)\overline{\psi(z')}
\, \diff\kappa_\La(z-z')\diff m(z).
\]
This is the so-called reduced covariance measure.
It is often more convenient to write $\kappa_\La=\tau_\La+c_\La\delta_0$. 
Indeed, for the standard point processes
that we have in mind (the Ginibre ensemble, 
and the zero set of the GEF)
the measure $\tau_\La$ has a density $k_\La$ with respect
to planar Lebesgue measure, called the \emph{(truncated) two-point function}.
For the Ginibre ensemble, the two-point function takes
the particularly simple form $k_\La(t)=-\pi^{-2}e^{-\pi\,t^2}$, while
for the zeros of the GEF it is given by
\[
k_\La(t)=\frac12\, \frac{{\rm d}^2}{{\rm d}t^2}\, t^2 (\coth t - 1)\,.
\]
The spectral measure $\rho_\La$ is the Fourier transform of $\kappa_\La$, and 
it satisfies the Plancherel-type identity
\begin{equation}\label{eq21}
\Cov  \left[n_\La (\varphi), n_\La(\psi)\right]
= \int_{\C} \widehat\varphi(\xi)
\overline{\widehat{\psi}(\xi)}\,{\diff}\rho_{\La}(\xi)
= \langle \widehat{\varphi}, \widehat{\psi} \rangle_{L^2(\rho_\La)}\,,
\end{equation}
for test functions $\varphi,\psi\in \mathfrak D$. 
This relation extends to $\varphi,\psi\in \mathcal S$; 
see Remark~2.1 in Part {\rm I}.

Under the spectral condition \eqref{eq:spectral-cond}, 
we define the stationary random field $V_\La$ by \eqref{eq:V-def}  
(see also \S5.1, Part {\rm I}). 
The spectral measure 
$\rho_{V_\La}$ of $V_\La$ is given by
\[
{\rm d}\rho_{V_\La}(\xi) = \done_{\C\setminus \{0\}}(\xi)\,
\frac{{\rm d}\rho_\La (\xi)}{|\xi|^2},
\]
see Theorem~5.8 in Part {\rm I}.

\subsection{Admissible measures}
For most of the article, it will be convenient to replace
the current $\nu_\Gamma(f\diff z)=\int_\Gamma f(z){\diff z}$ 
with the more general observables
\[
\mu(f)\ed \int f\, {\diff \mu}
\]
where $\mu$ is a complex-valued measure of finite total variation.
The weak Ahlfors regularity condition then corresponds
to the following notion.

\begin{defn}[Admissible complex-valued measures]
\label{def:admissibility}
We say that a compactly
supported complex-valued Borel measure $\mu$ on $\C$ 
of finite total variation is \emph{admissible} 
if there exists a constant $C=C_\mu$
such that
\[
|\mu(\D(x,r))|\le C r
\] 
for all $x\in\C$ and $r>0$.
\end{defn}

Denote by $\phi_\epsilon$ the Gaussian 
\begin{equation}
\label{eq:Gaussian}
\phi_\epsilon(z)=\frac{1}{\pi\epsilon^2}e^{-|z|^2/\epsilon^2}.
\end{equation}

\begin{claim}
\label{claim:preserve-Ahlfors}
Assume that $\mu$ is admissible with admissibility 
constant $C_\mu$, and define a regularized measure by 
$\mu_\epsilon\ed \phi_\epsilon*\mu$. 
Then $\mu_\epsilon$ is admissible with the 
same constant $C_\mu$.
\end{claim}

\begin{proof}
From the definition of $\mu_\epsilon$, Fubini's theorem, 
and a linear change of variables 
we get
\begin{align}
\mu_\epsilon(\D(x,r))
&=\int_{\D(x,r)}\int_{\C}\phi_\epsilon(z-w) \, {\diff \mu}(w)\diff m(z)\\
&=\iint_{\C\times \C}\done_{\D(x,r)}(\zeta+w)
\phi_{\epsilon}(\zeta) \, {\diff \mu}(w)\diff m(\zeta)\\
&=\int_{\C}\mu\left(\D(x+\zeta,r)\right)\phi_\epsilon(\zeta) \,\diff m(\zeta).
\end{align}
Since $\phi_\epsilon\ge 0$ with 
$\lVert \phi_\epsilon\rVert_{L^1(\C,\diff m)}=1$, 
this gives the upper bound
\[
\frac{\left|\mu_\epsilon(\D(x,r))\right|}{r}\le 
\int_{\C}\frac{|\mu\left(\D(x+\zeta,r)\right)|}{r}
\phi_\epsilon(\zeta) \, \diff \zeta\le C_\mu
\]
as claimed.
\end{proof}

\section{Geometric lemmas}
\label{s:geom}
\subsection{Rectifiable Jordan curves are weakly Ahlfors regular}
We supply a simple proof that any Jordan curve is weakly Ahlfors regular.

\begin{lem}
\label{lem:weak-Ahlfors}
Every rectifiable Jordan curve $\Gamma$ is weakly Ahlfors regular. In fact,
for any disk $D$ we have the upper bound
\begin{equation}
\label{eq:weak-reg}
\left|\int_{\Gamma\cap D}{\diff z}\right|
\le |\partial D|.
\end{equation}
\end{lem}

\begin{figure}[t!]
\vspace{-14pt}
\centering
\includegraphics[width=.55\textwidth]{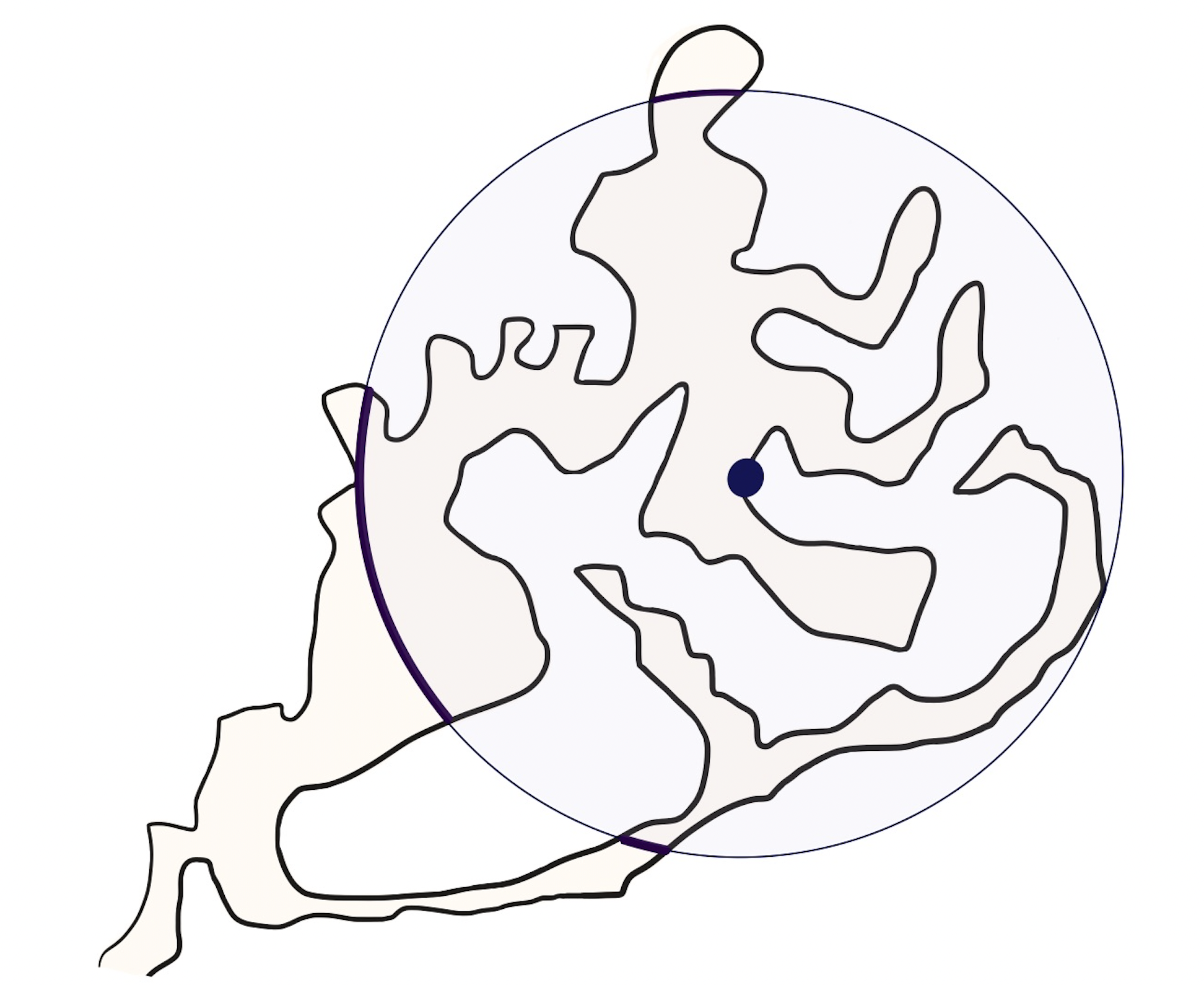}
\caption{Figure pertaining to the proof that Jordan curves
are weakly Ahlfors regular (Lemma~\ref{lem:weak-Ahlfors}). 
The union $\calC$ of circular arcs intersecting the 
Jordan domain $\calG$
is shown in purple, and the domains enclosed by
$(\Gamma\cap D)\cup \calC$ are lightly shaded.}
\label{fig:blowup}
\end{figure}

\begin{proof}
Denote by $\calG$ the Jordan domain enclosed by $\Gamma$ 
and fix an open disk $D$. We first assume that $\Gamma$ is analytic, which implies that
either $\Gamma=\partial D$, or it has finitely many intersections 
with $\partial D$. Indeed, after a translation and a rescaling,
we may assume that $D$ is the unit disk. If $\phi:\D\to\calG$ is a conformal map,
it extends conformally past the boundary, and $f$ is real-analytic
on the circle. 
If $\#(\Gamma\cap\partial D)$ were infinite, then the zeros
of the real-analytic
function $f(s)\ed |\phi(e^{{\rm i}s})|^2-1$ would have an accumulation point, which
implies that $f(s)=0$, so $\Gamma$ coincides with a circle.
Hence, $\calG\cap D$ is a finite union of Jordan domains with piecewise
analytic boundaries.
Moreover, $\calC\ed \overline{\calG\cap \partial D}$ is a disjoint finite union 
of circular arcs. If these are given the positive orientation, the domain $\calG$ always
remains on the left-hand side as we traverse $\mathcal{C}$, and we find that the curve
\[
\left(\Gamma\cap D\right)\cup \calC
\]
is a finite union of piecewise smooth Jordan curves, each of which bounds 
a connected component of $\G\cap D$; see Figure~\ref{fig:blowup}.
Cauchy's theorem then gives that
\begin{equation}
\label{eq:weak-ahlfors-analytic-jordan1}
\int_{\Gamma\cap D}{\diff z}=
-\int_{\calC}{\diff z},
\end{equation}
and we clearly have that
\begin{equation}
\label{eq:weak-ahlfors-analytic-jordan2}
\left|\int_{\calC}{\diff z}\right|\le
\int_{\calC}|\diff z|=|\partial D|.
\end{equation}
Hence, the proof is complete for analytic curves.

For a general rectifiable Jordan curve $\Gamma$, we argue by approximation.
By Carathéodory's theorem, there exists a conformal mapping $\phi$ of $\D$ onto $\calG$,
such that $\phi$ extends to a homeomorphism 
$\phi:\overline{\D}\to \overline{\G}$.
We obtain the desired approximation by taking
$\Gamma_t$ to be the curve parametrized by $\phi(t e^{{\rm i}s})$, $s\in[0,2\pi)$,
for $0<t<1$ and letting $t\to 1$. 
Since $\Gamma$ is rectifiable, a theorem of Riesz and Privalov (see \cite[6.8]{Pommerenke})
asserts that the derivative $\phi'(z)$ belongs to
the Hardy space $H^1(\T)$. In particular, the radial limit
$\phi'=\lim_{t\to 1}\phi'_t$ exists a.e.\ on $\T$, and $\phi_t'\to\phi'$ in $L^1(\T)$.
Since
$\phi(e^{{\rm i}s})$ supplies an absolutely continuous
parametrization of $\Gamma$, the quantity of interest may be expressed as
\[
\int_{\Gamma\cap D}{\rm d}z=\int_0^{2\pi}\done_{D}\big(\phi(e^{{\rm i}s})\big)
\partial_s\phi(e^{{\rm i}s})\,{\rm d}s=\int_0^{2\pi}\done_{D}\big(\phi(e^{{\rm i}s})\big)
{\rm i}e^{{\rm i}s}\phi'(e^{{\rm i}s})\,{\rm d}s.
\]
Next, note that $\done_{D}\circ \phi_t\to \done_{D}\circ \phi$ holds
everywhere on the circle. 
Indeed, since $D$ is open, any point $\phi(e^{{\rm i}s})\in D$ has a neighborhood
$V\subset D$. But since $\phi_t(e^{{\rm i}s})\to \phi(e^{{\rm i}s})$
for any $s$, we have $\phi_t(e^{{\rm i}s})\in V$ for $t$ sufficiently close to $1$.

As a consequence of these properties of the approximating parametrization
we find that for any fixed disk $D$, 
we have
\[
\lim_{t\to 1}\int_{0}^{2\pi}\done_{D}\big(\phi_t(e^{{\rm i}s})\big)
{\rm i}e^{{\rm i}s}\phi_t'(e^{{\rm i}s})\,{\rm d}s=
\int_{0}^{2\pi}\done_{D}\big(\phi(e^{{\rm i}s})\big){\rm i}e^{{\rm i}s}
\phi'(e^{{\rm i}s})\,{\rm d}s,
\]
or, in other words,
\[
\lim_{t\to 1 }\int_{\Gamma_t\cap D}{\rm d}z=\int_{\Gamma\cap D}{\rm d}z.
\]
Hence, for any given disk $D$ and any 
fixed $\epsilon>0$, 
we let $t$ be sufficiently close to $1$ so that
\[
\left|\int_{\Gamma\cap D}{\rm d}z-\int_{\Gamma_t\cap D}{\rm d}z\right|\le 
\epsilon|\partial D|.
\]
But since $\Gamma_t$ is analytic, it follows from \eqref{eq:weak-ahlfors-analytic-jordan1} and
\eqref{eq:weak-ahlfors-analytic-jordan2} and the reverse triangle inequality that
\[
\left|\int_{\Gamma\cap D}{\rm d}z\right|\le (1+\epsilon)|\partial D|.
\]
Since $\epsilon>0$ was arbitrary, the claim follows.
\end{proof}

If a rectifiable Jordan arc $\Gamma$ can
be completed to a rectifiable Jordan curve by appending 
a weakly Ahlfors regular Jordan arc, then $\Gamma$ is also
weakly Ahlfors regular; cf.\ Figure~\ref{fig:complementary}.  
The existence of such an arc would be guaranteed if, e.g.,
$\Gamma$ does not wind too wildly near any of its endpoints. Hence, at least for
rectifiable Jordan arcs, only the behavior near
the end-points matter for weak regularity.

\subsection{The density of \texorpdfstring{$\nu_{\Gamma}$}{nu-Gamma}}
\label{s:net-tangent}
We denote by $\gamma:I\to \Gamma$ the arc-length 
parametrization of $\Gamma$, and introduce the (\say{net tangent})
function
\begin{equation}
\label{eq:net-tangent}
\tau(x)=\tau_\Gamma(x)\ed\sum_{t\in\gamma^{-1}(x)}\gamma'(t).
\end{equation}
Let $\nu_\Gamma$ be the complex-valued measure given by the integration
current over $\Gamma$. That is,
\begin{equation}
\label{eq:int-h}
\int h\,{\rm d}\nu_{\Gamma}=\int_{\Gamma}h(z)\,{\rm d}z=\int_{I}h(\gamma(t)) \, \gamma'(t)\,{\rm}\diff t.
\end{equation}
Then, for any Borel set $B$, $\displaystyle \nu_\Gamma(B)=\int_{\Gamma\cap B}{\rm d}z$.

Recall that $\mathcal{H}^1$ is the one-dimensional Hausdorff measure on $\C$.

\begin{figure}[t!]
\centering
\includegraphics[width=.5\linewidth]{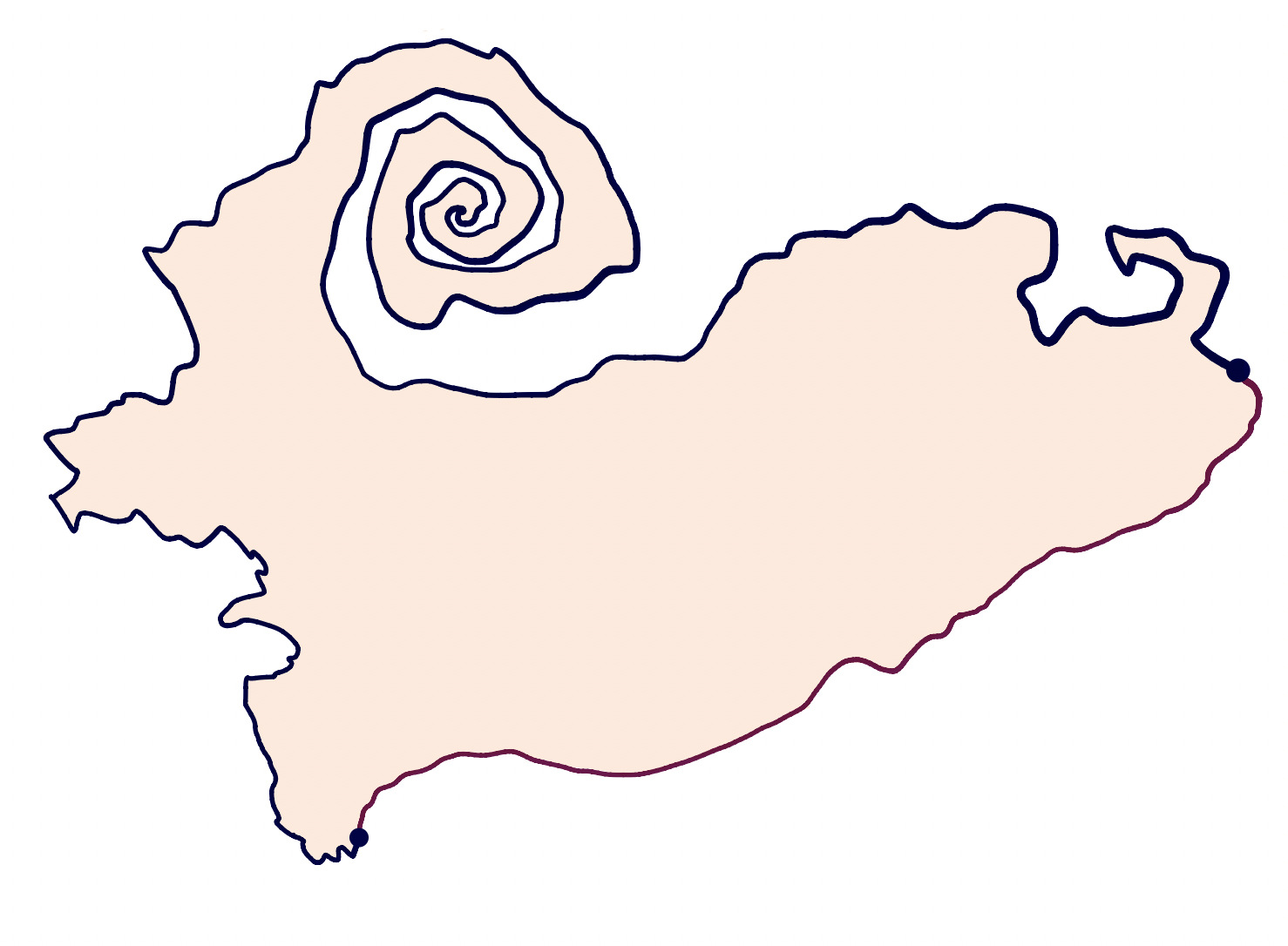}
\caption{A Jordan arc $\Gamma$ (black) which can be completed to a Jordan
curve by appending
a weakly Ahlfors regular arc (red), implying that $\Gamma$ is weakly Ahlfors; cf.\
Lemma~\ref{lem:weak-Ahlfors}.}
\label{fig:complementary}
\end{figure}

\begin{lem}
\label{lem:good-asymp}
Assume that $\Gamma$ is a weakly Ahlfors regular rectifiable curve. 
Then for $\mathcal{H}^1$-a.e.\ $x$, it holds that
\[
\lim_{\epsilon\to 0}\frac{1}{\epsilon}\int_{0}^\epsilon
\frac{\nu_{\Gamma}\left(\D(x,r)\right)}{2r}\diff r=\tau(x).
\]
\end{lem}

\begin{proof}
From \eqref{eq:int-h}, it is evident that $\nu_{\Gamma}$ is
absolutely continuous with respect to arc-length measure 
on $\Gamma$, which in turn is absolutely continuous with
respect to $\mathcal{H}^1\big\vert_{\Gamma}$. 
In fact, by applying the change of 
variables formula in \cite[Theorem 3.9]{EvansGariepy2015} to the 
functions $g=(h\circ \gamma)\,\gamma'$ and 
$f=\gamma$ (so that the Jacobian satisfies $Jf=1$ a.e.), 
we find that 
\[
\int_{I}h(\gamma(t))\gamma'(t)\, \diff t=
\int_{\Gamma}h(x)\left(\sum_{t\in \gamma^{-1}(x)}
\gamma'(t)\right)\diff \mathcal{H}^1(x)
=\int_{\Gamma}h(x)\tau(x)\, \diff \mathcal{H}^1(x),
\]
so that $\diff \nu_{\Gamma}(x)=\tau(x)\, \diff\mathcal{H}^1\big\vert_{\Gamma}(x)$.
In view of the upper bound
\begin{align*}
\int_{\Gamma}|\tau(x)|\,\diff\mathcal{H}^1(x)&\le 
\int_{\Gamma}\#\big\{t\in\gamma^{-1}(x)\big\}\diff\mathcal{H}^1(x)\\
&=\int_I |\gamma'(t)|\,\diff t
=|\Gamma|,
\end{align*}
and of the rectifiability of $\Gamma$,
we have $\tau\in L^1(\mathcal{H}^1\vert_{\Gamma})$. 

We next claim that
\begin{equation}
\label{eq:density}
\frac{\mathcal{H}^1\vert_{\Gamma}(\D(x,r))}{2r}\xrightarrow{r\to 0} 1.
\end{equation}
Indeed, the upper bound follows from the density bound \cite[Theorem~6.2]{MattilaBook}
for general rectifiable sets, and the lower bound is a consequence of 
the fact that $\Gamma$ has a tangent at $\mathcal{H}^1$-a.e.\ $x\in\Gamma$.

Since $\mathcal{H}^1\vert_{\Gamma}$ is a Borel regular measure 
(see \cite[p.\ 57]{MattilaBook}), 
we may apply the Lebesgue-Besicovitch differentiation 
theorem (\cite[Theorem 1.32]{EvansGariepy2015}) along with 
\eqref{eq:density}
to obtain
\begin{equation}
\label{eq:diff-measure-ball}
\lim_{r\to 0}
\frac{\nu_{\Gamma}(\D(x,r))}{2r}= 
\lim_{r\to 0}
\frac{1}{\mathcal{H}^1\vert_{\Gamma}(\D(x,r))}
\int_{\Gamma\cap \D(x,r)}\tau(x)
\diff\mathcal{H}^1(x)
\cdot \frac{\mathcal{H}^1\vert_{\Gamma}(\D(x,r))}{2r}
=\tau(x)
\end{equation}
for $\mathcal{H}^1$-a.e.\ $x\in\C$.

Observe next that
\[
\frac{1}{\epsilon}\int_0^\epsilon \frac{\nu_{\Gamma}
\left(\D(x,r)\right)}{2r}\diff r=\int_0^1 
\frac{\nu_{\Gamma}\left(\D(x,\epsilon t)\right)}{2\epsilon t}\diff t,
\]
and in view of the weak Ahlfors regularity
of $\Gamma$, the integrand on the right-hand side is bounded above by $1$.
Hence, the pointwise convergence \eqref{eq:diff-measure-ball} and the 
bounded convergence theorem together give that
\[
\lim_{\epsilon\to0}\frac{1}{\epsilon}\int_0^\epsilon 
\frac{\nu_{\Gamma}\left(\D(x,r)\right)}{2r}\diff r
=\int_0^1 \tau(x)\diff t=\tau(x)
\]
for $\mathcal{H}^1$-a.e.\ $x\in\C$.
This completes the proof.
\end{proof}

\begin{rem}\label{rem:finite-signed-length}
In view of Lemma~\ref{lem:good-asymp}, the 
net tangent $\tau_\Gamma$ of the weakly Ahlfors regular rectifiable curve $\Gamma$
(i.e., the density of ${\diff z}$ along 
$\Gamma$ with respect to Hausdorff measure $\mathcal{H}^1$) belongs 
to $L^\infty(\C,\diff \mathcal{H}^1)$. 

Furthermore, since $\Gamma_1$ 
and $\Gamma_2$ are assumed to be rectifiable, 
the angle between the tangents is either $0$ or $\pi$ 
at $\mathcal{H}^1$-a.e.\ point
where the curves intersect. Indeed, the set of points 
in the intersection for which neither of the curves
have a unimodular tangent has $\mathcal{H}^1$-measure $0$.
Furthermore, for any point in $\Gamma_1\cap\Gamma_2$ where both curves have
unimodular tangent and the angle of intersection is not $0$ or $\pi$,
there exists a punctured neighborhood where the curves do not intersect.
Therefore, the set of such points is at most countable.

Combining both observations with the 
definition~\eqref{eq:signed-length} of $\calL(\Gamma_1,\Gamma_2)$ 
and formula~\eqref{eq:net-tangent} for the net tangent $\tau_\La$, we arrive at the formula
\begin{align}
\mathcal{L}(\Gamma_1,\Gamma_2)&=
\int_{\C}\bigg(\sum_{s\in \gamma_1^{-1}(z),\; t\in\gamma_2^{-1}(z)}
\Re\big(\gamma_1^\prime(s) \,\overline{\gamma_2^\prime(t)}\big)\bigg)\,\diff\mathcal{H}^1(z)
\\
&=\int_{\C}\bigg(\sum_{s\in \gamma_1^{-1}(z),\; t\in\gamma_2^{-1}(z)}
\gamma_1^\prime(s)\,\overline{\gamma_2^\prime(t)}\bigg)\,\diff\mathcal{H}^1(z)\\
&=\int_{\Gamma_1\cap\Gamma_2}\tau_{\Gamma_1}(x)
\overline{\tau_{\Gamma_2}(x)} \, \diff \mathcal{H}^1(x) \, .
\end{align}
In particular, we see that the signed length $\mathcal{L}(\Gamma_1,\Gamma_2)$ 
is finite whenever $\Gamma_1$ and $\Gamma_2$ are weakly Ahlfors 
regular rectifiable curves.
\end{rem}

\section{The covariance structure of \texorpdfstring{$\act_\Lambda(\Gamma)$}{}}
\label{s:cov-flux}

The purpose of this section is to establish the basic formula
\eqref{eq:var-formula} for the covariance of the action of 
$V_\Lambda$ along two weakly Ahlfors regular rectifiable curves. 
The starting point is the following result taken from 
Proposition~5.9 and Remark~5.11 in Part {\rm I}.

\begin{lem}
\label{lem:var-V}
Assume that $(1+t^2)k_\La\in L^1(\R_{\ge 0}, \diff t)$.
Then for $\varphi, \psi\in\calS$, we have
\begin{equation}
\label{eq:covariance_of_vector_field}
\Cov  \left[V_\Lambda(\varphi),V_\Lambda(\psi)\right]
=\iint_{\C\times\C}\varphi(x)\overline{\psi(y)}K(|x-y|) \, \diff m(x) \, \diff m(y),
\end{equation}
where the kernel $K=K_{V_\La}$ (the two-point function for $V_\La$) is given by
\[
K(z)=-4\pi^2\int_0^\infty\log_+
\left(\frac{r}{|z|}\right)k_\La(r)\,r \,\diff r, \qquad z\in\C\setminus\{0\}.
\]
\end{lem}

Juxtaposing~\eqref{eq:covariance_of_vector_field} with the formula for the covariance on the 
Fourier side: 
\begin{equation}
\Cov  \left[V_\Lambda(\varphi),V_\Lambda(\psi)\right]
= \int_{\C} \widehat{\varphi}(\xi) \, \overline{\widehat{\psi}(\xi)} \, 
\frac{{\rm d}\rho_{\La}(\xi)}{|\xi|^2} 
\end{equation}
(see Theorem~5.8, Part~{\rm I}), we conclude that
\begin{equation}
\label{eq:fourier_transform_of_K}
\widehat{K}(\xi) = |\xi|^{-2} h_{\La}\, ,
\end{equation}
where $h_\La$ is the (radial) density of $\rho_\La$. The Fourier 
transform in~\eqref{eq:fourier_transform_of_K} should be understood 
in the sense of distributions, or alternatively in the sense of 
Fourier transform acting on $L^2(\C,m)$ functions.

For the remainder of this section, we will work in somewhat greater 
generally with observables 
$\mu(V_\Lambda)$, where $\mu$ is an admissible measure (recall
Definition~\ref{def:admissibility} above). 
In order to show that the variance of $\mu(V_\Lambda)$ is well defined,
we will approximate $\mu$ by $\mu_\epsilon=\mu*\phi_\epsilon$,
where $\phi_\epsilon(z)$ are the Gaussians from 
\eqref{eq:Gaussian}. 

\begin{lem}
\label{lem:cov-flux}
Assume that $\mu$ and $\nu$ are admissible complex-valued measures.
Then we have $\mu(V_\Lambda), \nu(V_\Lambda)\in L^2(\Omega,\P)$, and
\[
\Cov  \left[\int_{\C}V_\Lambda(z) \, {\diff \mu}(z),
\int_{\C}V_\Lambda(z) \, \diff \nu(z)\right]
=\lim_{\epsilon\to 0}\iint_{\C\times\C}K(|x-y|) \, {\diff \mu}_\epsilon(x)
\, \diff \bar{\nu}_\epsilon(y).
\]
\end{lem}

\begin{proof}
It will suffice to show that 
$\mu(V_\Lambda)=\int V_\Lambda {\diff \mu}\in L^2(\Omega,\P)$
and that
\[
\Var \left[\mu(V_\Lambda)\right]
=\lim_{\epsilon\to0}\iint_{\C\times\C}K(|x-y|) \,
\diff \mu_\epsilon(x) \, \diff \bar{\mu}_\epsilon(y).
\]
Since $\mu_\epsilon$ is absolutely continuous with a density in 
$\calS$, by~\eqref{eq:fourier_transform_of_K} we have
\[
\Var \left[\mu_\epsilon(V_\Lambda)\right]
=\int_{\C}|\widehat{\mu}_\epsilon(\xi)|^2
\, \frac{\diff \rho_\La(\xi)}{|\xi|^2}
=\int_{\C}e^{-\epsilon^2|\xi|^2}|\widehat{\mu}(\xi)|^2
\, \frac{\diff \rho_\La(\xi)}{|\xi|^2}.
\]
Notice that the integrand is monotonically 
increasing as $\epsilon$ decreases to $0$, 
so it will be sufficient to establish that
\begin{equation}
\label{eq:property-1}
\sup_{\epsilon>0}\int_{\C}
e^{-\epsilon^2|\xi|^2}|\widehat{\mu}(\xi)|^2
\, \frac{\diff \rho_\La(\xi)}{|\xi|^2}<\infty.
\end{equation}
Indeed, it will then follow from monotone convergence 
that $\Var \left[\mu_\epsilon(V_\Lambda)\right]$
converges to $\lVert \widehat{\mu}\rVert_{L^2(|\xi|^{-2}\rho_\La)}^2$ as 
$\epsilon\to 0$, which implies that
\[
\lVert \widehat{\mu}_\epsilon
-\widehat{\mu}_\delta\rVert_{L^2(\C,|\xi|^{-2}\rho_\La)}^2
=\int_{\C}\Big(e^{-\epsilon^2|\xi|^2}+e^{-\delta^2|\xi|^2}
-2e^{-\frac{\epsilon^2+\delta^2}{2}|\xi|^2}\Big)
|\widehat{\mu}(\xi)|^2 \, \frac{\diff\rho_\La(\xi)}{|\xi|^2}\to 0
\]
as $\epsilon,\delta\to 0$.
Hence, we see that $\displaystyle \mu(V_\La)=\lim_{\epsilon\to 0} \, 
\mu_\epsilon (V_\La)$ in $L^2(\Omega,\bP)$,
and the variance
may be expressed as
\[
\Var [\mu(V_\Lambda)]=\lim_{\epsilon\to 0}\iint_{\C\times\C}
K(|x-y|) \, {\diff \mu}_\epsilon(x) \, \diff \bar{\mu}_\epsilon(y).
\]
In order to see why \eqref{eq:property-1} holds,
notice that by Lemma~\ref{lem:var-V} we have
\begin{align}
\label{eq:var-before-Fubini}
\Var \left[\mu_\epsilon(V_\Lambda)\right]
&=\iint_{\C\times\C}K(|x-y|) \, {\diff \mu}_\epsilon(x)
\, \diff\bar{\mu}_\epsilon(y)\\
&=-4\pi^2 \iint_{\C\times\C}
\left(\int_{0}^\infty \log_+\frac{t}{|x-y|}k_\La(t)\,t\,\diff t\right)
{\diff \mu}_\epsilon(x) \, \diff \bar{\mu}_\epsilon(y).
\end{align}
Note that the measure $|\mu_\epsilon|$ is absolutely continuous
with density in $\calS$, so we may integrate against $|\mu_\epsilon|\times|\mu_\epsilon|$
to get
\begin{multline}
\iint_{\C\times\C}\int_0^\infty\log_+\frac{t}{|x-y|}|k_\La(t)|\,t\,{\rm d}|\mu_\epsilon|(x)
\, {\rm d}|\mu_\epsilon|(y) \, {\rm dt}\\
\le \iint_{\C\times\C}\int_0^\infty\left(\big|\log|x-y|\big|+\big|\log t\big|\right)
|k_\La(t)|\,t\,{\rm d}|\mu_\epsilon|(x)
\, {\rm d}|\mu_\epsilon|(y) \, {\rm dt}\\
\lesssim_\epsilon \int_0^\infty \big(1+|\log t|\big)|k_\La(t)|\,t\,\diff t<\infty.
\end{multline}
We may thus apply Fubini's theorem to the right-hand side of \eqref{eq:var-before-Fubini}
to get
\[
\Var \left[\mu_\epsilon(V_\Lambda)\right]=
-4\pi^2\int_{\C}\int_0^\infty t^2k_\La(t)I_\epsilon(t,x)\,\diff t \, \diff\mu_\epsilon(x),
\]
where
\[
I_\epsilon(t,x)
=\frac{1}{t}\int_\C\log_+\frac{t}{|x-y|}\,\diff\bar\mu_\epsilon(y)
=\frac{1}{t}\int_0^\infty\bar{\mu}_\epsilon\left(\D(x,te^{-s})\right) \, \diff s,
\]
the last equality being a consequence of the 
\say{layer cake formula}; that is, integration with 
respect to the distribution function. 
Hence, by the triangle inequality, we get that
\[
\Var \left[\mu_\epsilon(V_\Lambda)\right]
\le 4\pi^2\int_{\C}\int_0^\infty t^2\,|k_\La(t)|\, 
|I_\epsilon(t,x)| \, \diff t\,\diff |\mu_\epsilon|(x).
\]
By Claim~\ref{claim:preserve-Ahlfors}, the measure $\mu_\epsilon$
is admissible, with admissibility constant independent of $\epsilon$.
Therefore,
\[
|I_\epsilon(t,x)|\le \frac{1}{t}\int_0^t
\frac{|\mu_\epsilon\left(\D(x,r)\right)|}{r}\diff r\le C_\mu,
\]
and, as a consequence, 
\[
\Var \left[\mu_\epsilon(V_\Lambda)\right]
\le 
C_\mu|\mu_\epsilon|(\C)\int_0^\infty t^2\,|k_\La(t)| \, \diff t,
\]
which is readily seen to be bounded above independently
of $\epsilon$ by use of the trivial bound $|\mu*\phi_\epsilon|(\C)
\le|\mu|(\C)\int_\C\phi_\epsilon(z)\diff m(z)=|\mu|(\C)$.
\end{proof}

Since in the end we prefer to think in 
terms of the principal value integral,
we should check that both regularizations of the integral 
$\displaystyle{\iint_{\C\times \C}K(|x-y|) \, \diff\mu(x) \, 
\diff\bar\mu(y)}$ give the same result.

\begin{lem}\label{lem:cov-flux-pv}
For any two admissible measures $\mu$ and $\nu$, we have that
\begin{align*}
\label{eq:property-2}
\lim_{\epsilon\to 0}\iint_{\C\times\C}
K(|x-y|) \, {\diff \mu}_\epsilon(x) \, \diff \bar{\nu}_\epsilon(y)
&=\lim_{\epsilon\to 0}\iint_{\C\times\C} \done_{\{|x-y|>\epsilon\}}
K(|x-y|)\, \diff\mu(x) \, \diff\bar{\nu}(y)
\\
&\ed\pv \iint_{\C\times\C}K(|x-y|) \, {\diff \mu}(x) \, \diff \bar{\nu}(y).
\end{align*}
\end{lem}

\begin{proof}
Again, by polarization, it suffices to check the condition for $\mu=\nu$.
Recall that the convolution of
two centered Gaussians is again a centered Gaussian; 
\begin{equation}
\label{eq:convol-Gaussian-factor}
\phi_\epsilon*\phi_\epsilon=\phi_{\sqrt{2}\epsilon}.
\end{equation}
Thinking of $K$ as a function in $\C$ by the identification $K(z)=K(|z|)$,
we rewrite the regularized variance as
\begin{align}
\iint_{\C\times \C} K(|x-y|)\, {\diff \mu}_\epsilon(x) \, \diff \bar{\mu}_\epsilon(y)
&= \int_\C [K*\phi_\epsilon*\bar\mu]\,[\phi_\epsilon*\mu] \, \diff m\\
&= \int_\C [K*\phi_\epsilon*\bar\mu]*\phi_\epsilon\,{\diff \mu}
\end{align}
where we have used the general distributional identity
\[
\left(f, g*h\right)=\left( f*g, h\right)
\]
applied to $f=K*\phi_\epsilon$, $g=\phi_\epsilon$ and $h=\mu$
to arrive at the last equality. Using the associativity
of convolution along with the identity 
\eqref{eq:convol-Gaussian-factor}, we recognize this as
\[
\int_\C [K*\phi_\epsilon*\bar\mu]*\phi_\epsilon\,{\diff \mu}
=\int_{\C} \left(\int_{\C}K(|x-y|)
\, \diff \bar\mu_{\sqrt{2}\epsilon}(y)\right){\diff \mu}(x).
\]
The quantity of interest is thus
\begin{multline}
\int_{\C}\bigg(\int_{\C}K(|x-y|)
\, \diff \bar{\mu}_{\sqrt{2}\epsilon}(y)-\int_{|x-y|>\epsilon}
K(|x-y|) \, \diff \bar\mu(y)\bigg) {\diff \mu}(x)\\
=-4\pi^2\int_{\C}\int_0^\infty k_\La(t)t^2 I^\Delta_\epsilon(t,x)\, 
\diff t \, \diff\mu_\epsilon(x)
\end{multline}
where
\begin{align}
I^\Delta_\epsilon(t,x)&=\frac{1}{t}\int_0^\infty
\left[\mu_{\sqrt{2}\epsilon}\left(\D(x,t e^{-s})\right)
-\mu\left(\D(x,te^{-s})\setminus \D(x,\epsilon)\right) \right]\diff s\\
&=\frac{1}{t} \int_0^t\left(\frac{\mu_{\sqrt{2}\epsilon}\left(\D(x,r)\right)}{r}
-\frac{\mu\left(\D(x,r)\setminus \D(x,\epsilon)\right)}{r}\right)\diff r.
\end{align}

By the admissibility assumption and by the fact that convolution
preserves weak Ahlfors regularity (Claim~\ref{claim:preserve-Ahlfors}), the integrand on 
the right-hand side is uniformly bounded, as
\[
\bigg|\frac{\mu_{\sqrt{2}\epsilon}\left(\D(x,r)\right)}{r}
-\frac{\mu\left(\D(x,r)\setminus \D(x,\epsilon)\right)}{r}\bigg| 
\le \frac{|\mu_{\sqrt{2}\epsilon}(\D(x,r))|}{r} + \frac{|\mu(\D(x,r))|}{r} \le 2 \, C_\mu\, .
\]
Therefore, we can apply the bounded convergence theorem for each fixed $t>0$ and get 
\[
\lim_{\epsilon\to 0}I^\Delta_\epsilon(t,x) = 0
\]
for $\mathcal{H}^1$-a.e.\ $x$. Hence, in view of the condition that
$t^2 k_\La(t)\in L^1(\R_{\ge 0},\diff t)$, the claim follows 
from the dominated convergence theorem.
\end{proof}

\section{Proof of the main results}
\label{s:pf-main}
\subsection{The asymptotic covariance structure of the electric action}
Recall that Theorem~\ref{thm:main} asserts that if the 
two-point function $k_\La$ is radial and satisfies 
$(1+t^2)k_\La(t)\in L^1(\R_{\ge 0},\diff t)$, then
for any weakly Ahlfors regular rectifiable curves $\Gamma_1$ and $\Gamma_2$ we have
\[
\Cov\left[\int_{R\Gamma_1} V_\La(z)\,\diff z,\int_{R\Gamma_2} V_\La(z)\,\diff z\right]=
R\left(C_\La+o(1)\right)\calL(\Gamma_1,\Gamma_2)
\]
as $R\to\infty$, where $\displaystyle{C_\La=-8\pi^2\int_0^\infty k_\La(t)\, t^2\,\diff t}$.

\begin{proof}[Proof of Theorem~\ref{thm:main}]
We will show that for any two admissible measures 
$\mu$ and $\nu$ such that the limit
\begin{equation}
\label{eq:assumpt-conv-density}
\tau_\nu(x)\ed \lim_{\epsilon\to 0}\frac{1}{\epsilon}
\int_0^\epsilon\frac{\nu(\D(x,r))}{2r}\diff r
\end{equation} 
exists $\mathcal{H}^1$-a.e., we have
\begin{equation}
\label{eq:asymp-general}
\Cov  \left[\int_{\C} V_\Lambda(Rz)\,{\diff \mu}(z),
\int_{\C} V_\Lambda(Rz)\, \diff \nu(z)\right]
=\left(C_\Lambda+o(1)\right)\,R^{-1}\int_{\C}\overline{\tau_\nu(z)}\,{\diff \mu}(z).
\end{equation}
Theorem~\ref{thm:main} will follow by combining 
\eqref{eq:asymp-general} with Lemma~\ref{lem:good-asymp}, which states that
\eqref{eq:assumpt-conv-density} holds for the complex-valued measure $\nu_\Gamma$,
which was defined by
$\nu_\Gamma(f)=\int_\Gamma f\,\diff z$. 
This measure
is clearly admissible, since $\Gamma$
is assumed to be weakly Ahlfors regular.
We note that 
\[
\int_{\Gamma_1\cap\Gamma_2}\overline{\tau_{\Gamma_2}(x)}\,\diff\nu_{\Gamma_1}(x)
= \int_{\Gamma_1\cap \Gamma_2} \tau_{\Gamma_1}(x) \, \overline{\tau_{\Gamma_2}(x)} 
\, {\rm d} \mathcal{H}^1(x)=\mathcal{L}(\Gamma_1,\Gamma_2),
\]
see Remark~\ref{rem:finite-signed-length}.

Putting together the two formulas from Lemma~\ref{lem:cov-flux} and 
Lemma~\ref{lem:cov-flux-pv} for the regularized covariance, we have
\begin{align}
\Cov  \left[\int_{\C} V_\Lambda(R z) \, {\diff \mu}(z),
\int_{\C} V_\Lambda( Rz) \, \diff \nu(z)\right]
&=\pv \iint_{\C\times\C}K(R|x-y|) \, {\diff \mu}(x) \, \diff \bar{\nu}(y)\\
&=\lim_{\epsilon\to0}\iint_{\C\times\C}\done_{\{|x-y|>\epsilon\}}
K(R|x-y|) \, {\diff \mu}(x)\, \diff \bar{\nu}(y).
\end{align}
For any $\epsilon>0$, we may write the 
truncated covariance integral as
\begin{multline}
\iint_{\C\times\C}\done_{\{|x-y|>\epsilon\}}
K(R|x-y|)\, {\diff \mu}(x) \, \diff \bar{\nu}(y)\\
=-4 \pi^2\int_{\C}\int_{|x-y|>\epsilon}
\int_0^\infty \log_+\frac{t/R}{|x-y|}k_\La(t)\,t \, \diff t\,
{\diff \mu}(x) \, \diff \bar{\nu}(y).
\end{multline}
This integral is absolutely convergent, 
and an application of Fubini's theorem
gives that
\[
\iint_{\C\times\C}\done_{\{|x-y|>\epsilon\}}
K(R|x-y|) \, {\diff \mu}(x) \, \diff \bar{\nu}(y)\\
=\frac{1}{R}\int_{\C}I_{R,\epsilon}(x) \, {\diff \mu}(x)
\]
where
\begin{align}
\label{eq:int-IJ}
I_{R,\epsilon}(x)&=-4\pi^2\int_0^\infty k_\La(t)\,t^2
\left(\frac{R}{t}\int_{|x-y|>\epsilon}
\log_+\frac{t/R}{|x-y|}\, \diff \bar{\nu}(y)\right) \, \diff t\\
& 
\ed -4\pi^2\int_0^\infty k_\La(t)\,t^2
J_{R,t,\epsilon}(x) \diff t.
\end{align}
Integrating with respect to the distribution function, 
we arrive at
\begin{align}
J_{R,t,\epsilon}(x)&= \frac{R}{t}
\int_{|x-y|>\epsilon}\log_+\frac{t/R}{|x-y|}\diff \bar{\nu}(y)\\
&=\frac{1}{t/R}\int_{0}^\infty\bar{\nu}
\left(\left\{y: \epsilon<|x-y|<\frac{t}{R}e^{-s}\right\}\right)\diff s
\\ &=\frac{1}{t/R}\int_{\epsilon}^{t/R}
\frac{\bar{\nu}\left(\D(x,r)\setminus\D(x,\epsilon)\right)}{r}\diff r.
\end{align}
By the definition of admissible measures, we have
\[
|J_{R,t,\epsilon}(x)|\le C
\]
uniformly in $x,R$ $t$ and $\epsilon$. What is more, 
the existence of
$\displaystyle \lim_{\epsilon\to 0}J_{R,t,\epsilon}(x)=J_{R,t,0}(x)$ 
for all $x$ follows from admissibility,
while the assumption \eqref{eq:assumpt-conv-density} together with 
admissibility of $\nu$ ensure that for any fixed $t>0$, 
we have 
\[
J_{R,t,0}(x)=\frac{1}{t/R}\int_{0}^{t/R}
\frac{\bar{\nu}\left(\D(x,r)\right)}{r}
\diff r \xrightarrow{R\to \infty} 2\overline{\tau_\nu(x)}
\]
for $\mathcal{H}^1$-a.e.\ $x \in \C$. 
Hence, we may apply the dominated convergence theorem to the integral
\eqref{eq:int-IJ} to find
\begin{equation*}
\lim_{R\to\infty} \lim_{\epsilon\to 0}I_{R,\epsilon}(x)  =
-8\pi^2 \int_0^\infty k_\La(t)\,t^2 \diff t\;\overline{\tau_\nu(x)}
\ed C_\Lambda \overline{\tau_\nu(x)} \, .
\end{equation*}
We also have for free the bound
\[
|I_{R,\epsilon}(x)|\le \int_0^\infty t^2|k_\La(t)| \,
|J_{R,t,\epsilon}(x)|\diff t\le 
C\int_0^\infty t^2 k_\La(t)\diff t
\]
for $\mathcal{H}^1$-a.e. $x\in\C$.
If we decompose $\mu$ as $\mu=\mu^+_{1} -\mu^{-}_{1}
+{\rm i} \, \mu_2^+-{\rm i} \, \mu_2^{-}$ where each $\mu_i^\pm$ is a finite positive
measure, the bounded convergence theorem applied
to each of the four integrals then gives
\begin{align}
\pv \iint_{\C\times\C}
K(R|x-y|){\diff \mu}(x)\diff \bar{\nu}(y)
&=-\lim_{\epsilon\to 0}
\frac{1}{R}\int_{\C} I_{R,\epsilon}(x) \, {\diff \mu}(x)
\\&= \left(C_\Lambda+o(1)\right) R^{-1} \int_{\C} \overline{\tau_\nu(x)} \, {\diff \mu}(x)
\end{align}
as $R\to\infty$, as claimed.
\end{proof}

\begin{rem}
\label{rem:formula-CLambda}
The constant $C_\La$ in Theorem~\ref{thm:main} is given by
\[
C_\La=-8\pi^2\int_0^\infty k_\La(t)t^2\,\diff t.
\]
From this formula, it is not immediately clear that it is positive,
but another representation clarifies matters. Under the conditions of the theorem,
the spectral measure $\rho_\La$ has a radial density $h_\La$, and we have
\[
h_\La=c_\La+\widehat{k}_\La,
\]
where
\[
\widehat{k}_\Lambda(\xi)=\int_{\C}e^{-2\pi{\rm i}\xi\cdot z}k_\La(|z|)\,\diff m(z).
\]
Being the density of a positive measure, $h_\La$ is certainly positive. Moreover, the zeroth moment
condition \eqref{eq:zeroth-moment} gives that $h_\La(0)=0$. Since $(1+|z|)\,k_\La(|z|)$
belongs to $L^1(\C,\diff m)$, we also have $h_\La(|z|)\in C^1(\C)$. Then,
by positivity, we find that $h'_\La(0)=0$ as well.
Since moreover $\lVert h_\La\rVert_{L^\infty}\le |c_\La|
+\lVert k_\La(|\cdot|)\rVert_{L^1(\C)}$, we have
\[
\lim_{\tau\to 0}\frac{1}{\tau}h_\La(\tau)=\lim_{\tau\to\infty}\frac{1}{\tau}h_\La(\tau)=0,
\]
so integrating by parts we find that
\begin{align}
\int_0^\infty \frac{h_\La(\tau)}{\tau^2} \,\diff \tau
&=\int_0^\infty\frac{1}{\tau}\partial_\tau\left(\int_0^\infty\int_0^{2\pi}
e^{2\pi {\rm i}r\tau\cos\vartheta}k_\La(r)r\,\diff r\diff\vartheta\right)\diff\tau 
\\&=2\pi{\rm i}\int_0^\infty\frac{1}{\tau}\int_0^\infty\int_0^{2\pi} \cos\vartheta
e^{2\pi {\rm i}r\tau\cos\vartheta}r^2 k_\La(r)\,\diff r\diff\vartheta\diff\tau\\
&=-4\pi^2\int_0^\infty\int_0^\infty\frac{1}{\tau}J_1(2\pi r\tau) r^2 k_\La(r)\,\diff r\diff\tau,
\end{align}
where the last equality follows from the identity
\[
\int_0^{2\pi}\cos(\vartheta)e^{2\pi {\rm i}r\tau\vartheta}\diff\vartheta
=2\pi {\rm i}\,J_1(2\pi r\tau).
\]
But $\displaystyle{\int_0^\infty \frac{J_1(2\pi r\tau)}{\tau}\diff\tau=1}$, 
so an application of Fubini's theorem gives that
\[
\int_0^\infty \frac{h_\La(\tau)}{\tau^2}\,\diff \tau=-4\pi^2\int_0^\infty 
r^2 k_\La(r)\,\diff r.
\]
Since the left-hand side is clearly positive, it follows that $C_\La$ is positive as well.
\end{rem}

\subsection{The asymptotic charge fluctuations in rectifiable Jordan domains}
\begin{proof}[Proof of Theorem~\ref{thm:main2}]
We fix a Jordan domain $\calG$ with rectifiable boundary $\Gamma=\partial\calG$.
Recall that Lemma~\ref{lem:weak-Ahlfors} asserts that any rectifiable 
Jordan curve $\Gamma$ is weakly Ahlfors regular. Moreover, we have
\[
\Var\big[n_\La(R\calG)\big]
=\Var\big[\frac{1}{2{\rm i}}\int_{\partial (R\calG)}V_\La(z) \, \diff z\big]
=\frac14\Var\big[\act_\La(\partial(R\calG))\big].
\]
As a consequence, an application Theorem~\ref{thm:main}
with $\Gamma_1=\Gamma_2=\partial\calC$ gives
\[
\Var\big[n_\La(R\calG)\big]=\frac14 R\big(C_\La+o(1)\big) \, |\partial\calG|
\]
as $R\to\infty$, where $C_\La$ is the constant 
$\displaystyle{-8\pi^2\int_0^\infty k_\La(t)\,t^2\,\diff t}$.
This completes the proof.
\end{proof}

\subsection{Logarithmic asymptotics for the work}
The purpose of this section is prove Theorem~\ref{thm:flux}. That is,
we need to obtain the asymptotics of the variance
of the work
\[
\work_\Lambda(R\Gamma)=\Re\int_{R\Gamma} V_\Lambda(z)\,{\diff z},
\]
of $V_\Lambda$ along $R\Gamma$. We recall that for this result, 
we assume that $\Lambda$ is an invariant point process 
subject to the stronger moment condition 
$(1+t^3)k_\La(t)\in L^1(\R_{\ge 0},\diff t)$ for the two-point
function $k_\La$ of $\Lambda$.

Below, we will express the work $\work_\La(\Gamma)$
in terms of the increments 
\[
{\sf \Delta}_a\Pi_\La(z)\stackrel{\rm def}= \Pi_\La(z+a)-\Pi_\La(z),
\] 
where $\Pi_\La$ is the random potential for $\La$, i.e. the solution to
$\Delta \Pi_\La=2\pi (n_\La-c_\La\hspace{1pt}m)$, defined in
\S 6 in Part {\rm I}. The potential is given by $\Pi_\La (z)
\stackrel{\rm def}= \log|F_\La (z)| - \tfrac12\, \pi c_\La |z|^2$,
where $F_\La$ is the random entire function
\[
F_\La (z) = \exp\bigl[ -\Psi_1(\infty)z - \frac12\, \Psi_2(\infty)z^2 \bigr]\,
\prod_{|\la|<1} (\la-z) 
\prod_{|\la|\ge 1} \left(\frac{\la-z}\la\, 
\exp\Bigl[\,\frac{z}\la\, + \frac{z^2}{2\la^2}\, \Bigr]\right)\,,
\]
where for $j=1,2$, $$\Psi_j(\infty)=\lim_{R\to\infty}\sum_{1\le |\la|\le R} \frac{1}{\la^j}$$ with
convergence in $L^2(\Omega,\P)$. For $j=1$, this limit exists under the
spectral assumption~\eqref{eq:spectral-cond}; see Lemma 3.3 in Part {\rm I}. 
It is worth mentioning that a straightforward computation 
(see Theorem~6.2 in Part {\rm I}) shows that
\begin{equation*}
\label{eq:Delta-pot-a}
{\sf \Delta}_a\Pi(0)=\lim_{R\to\infty}\sum_{|\la|\le R}
\big(\log|a-\la|-\log|\la|\big)-\tfrac12 \pi c_\La |a|^2,
\end{equation*}
where the convergence is locally uniform in $a$.

Before we proceed, we derive a useful formula for the variance of
${\sf \Delta}_a\Pi(z)$. By Theorem~6.2 in Part {\rm I}, 
$\Pi_\La$ has stationary increments, so it suffices to analyze 
$\Var \big[{\sf \Delta}_a \Pi_\La(0)\big]$. 
\begin{claim}
\label{claim:1}
For any $a\in\C\setminus\{0\}$, we have
\begin{equation}
\label{eq:formula-DelPot-Phi}
\Var \big[{\sf \Delta}_a\Pi_\La(0)\big]
=\frac{1}{2\pi}\int_{\C}K(|s|)\Phi(a/s) \, {\diff}m(s)
\end{equation}
where
\begin{equation}
\label{eq:def_of_Phi}
\Phi(a/s)\ed 2\log|s|-\log|s-a|-\log|s+a|=
-\log\left|1-\frac{a^2}{s^2}\right|
\end{equation}
and where $K(s)=K(|s|)$ is the same as above (defined in Lemma~\ref{lem:var-V}).
\end{claim}
\begin{proof}
Let 
\[
\varphi_a(z)=\frac{1}{\pi}\Big(\frac{1}{\bar{z}-\bar{a}}
-\frac{1}{\bar{z}}\Big) \, 
\]
and note that $\partial \varphi_a=\delta_a-\delta_0$. As $\partial \, \Pi_\La = \frac{1}{2} V_\La$, we get 
\begin{equation}
\label{eq:Delta_Pi_in_terms_of_V}
{\sf\Delta}_a\Pi_\La(0)=\frac12 V_\La(\varphi_a) \, .
\end{equation}
By Remark~\ref{rem:formula-CLambda}, the spectral measure $\rho_{\La}$ 
is absolutely continuous with respect to Lebesgue measure and has a 
$C^1$-regular bounded (radial) density $h_\La$. Furthermore, a simple computation shows that
\[
\widehat{\varphi}_a(\xi)
=\big(1-e^{-2\pi{\rm i}a\cdot\xi}\big) \, \frac{1}{{\rm i}\overline{\xi}} \, ,
\]
in the sense of distributions, and Theorem~5.8 from Part~{\rm I} implies that
\begin{equation}
\label{eq:formula_for_variance_field_phi_a_fourier_side}
\Var \left[V_\La(\varphi_a)\right] = \int_{\C} \big|\widehat{\varphi}_a(\xi)\big|^2
\, \frac{{\diff}\rho_\La(\xi)}{|\xi|^2} = \int_{\C} \big|\widehat{\varphi}_a(\xi)\big|^2 \, 
\frac{h_\La(|\xi|)}{|\xi|^2} \,{\rm d}m(\xi) < \infty\, .
\end{equation}
In particular, by~\eqref{eq:Delta_Pi_in_terms_of_V} we see that 
${\sf\Delta}_a\Pi_\La(0) \in L^2(\Omega,\bP)$ and it remains to 
prove the formula for $\Var \big[{\sf \Delta}_a\Pi_\La(0)\big]$. 
By rotational invariance of $\La$, we may assume that $a>0$. 
Recall that the function $K$ is the (distributional) 
Fourier transform of $|\xi|^{-2} h_\La$, the spectral measure of $V_\La$, 
see~\eqref{eq:fourier_transform_of_K}. As both $|\widehat{\varphi}_a(\xi)|^2$ 
and $|\xi|^{-2} h_\La$ are in $L^2(\C,m)$, we may apply the 
Plancherel identity to~\eqref{eq:formula_for_variance_field_phi_a_fourier_side} and get that
\begin{equation}
\label{eq:variance_of_field_varphi_convoluted}
\Var \left[V_\La(\varphi_a)\right] = \int_{\C}\big(\varphi_a \ast \overline{\varphi}_a \big) (s) \,
K(|s|) \, {\diff}m(s) \, .
\end{equation}
To compute the convolution $\varphi_a\ast \overline\varphi_a$, observe that
\begin{equation*}
\big(\varphi_a \ast \overline{\varphi}_a \big) (s) = \frac{1}{\pi^2}\int_{\C}
\Big(\frac{1}{\bar{z}-a}-\frac{1}{\bar{z}}\Big)
\Big(\frac{1}{z-s-a}-\frac{1}{z-s}\Big) 
\, \diff m(z)\, ,
\end{equation*}
and by expanding the product we see that
\begin{multline}
\big(\varphi_a \ast \overline{\varphi}_a \big) (s) = 
\label{eq:inner-int}
\lim_{R\to\infty}\frac{1}{\pi^2}\int_{|z|\le R}\Big(\frac{1}{(\bar{z}-a)(z-a-s)}
-\frac{1}{(\bar{z}-a)(z-s)}\Big) {\diff}m(z)
\\+\lim_{R\to\infty} \frac{1}{\pi^2}\int_{|z|\le R}\Big(\frac{1}{\bar{z}(z-s)}
-\frac{1}{\bar{z} (z-a-s)}\Big) {\diff}m(z) \, .
\end{multline}
Now, for any $\alpha>0$ and $\beta\in\C$ we have that
\begin{align*}
\int_{|z|\le R}\frac{{\diff}m(z)}{(\bar{z}-\alpha)(z-\beta)}
&=2\int_{|z|\le R}\bar\partial\log|z-\alpha|
\, \frac{{\diff}m(z)}{z-\beta}
\\
&=2\pi\log|\beta-\alpha|-\frac{1}{{\rm i}}
\int_{|z|=R}\frac{\log|z-\alpha|}{z-\beta}\, {\diff}z\\
&=2\pi\log|\beta-\alpha|-2\pi\log R+o(1)
\end{align*}
as $R\to\infty$. Combining this with~\eqref{eq:inner-int}, we find that
\[
\big(\varphi_a \ast \overline{\varphi}_a \big) (s) 
= \frac{2}{\pi}\big(2\log|s|-\log|s-a|-\log|s+a|\big) 
\stackrel{\eqref{eq:def_of_Phi}}{=} \frac{2}{\pi} \Phi(a/s) \, .
\]
Plugging back into~\eqref{eq:variance_of_field_varphi_convoluted} gives
\[
\Var \left[V_\La(\varphi_a)\right] = \frac{2}{\pi}\int_{\C} K(|s|) \, \Phi(a/s) \, {\diff}m(s)
\]
which, together with~\eqref{eq:Delta_Pi_in_terms_of_V}, gives the claim.
\end{proof}

\begin{thm}
\label{thm:var-pot}
Assume that the truncated two-point function $k_\La$ satisfies 
$(1+t^3)k_\La(t)\in L^1(\R_{\ge 0}, {\diff}t)$. Then
\[
\Var\left[{\sf \Delta}_a\Pi_\La(0)\right]=\left(D_\La+o(1)\right)\log |a|,
\]
as $a\to\infty$, where
\[
D_\La=2\pi^2\int_{0}^\infty t^3 k_\La(t) \, {\diff}t.
\]
\end{thm}
\begin{proof}
By Claim~\ref{claim:1}, we have
\begin{equation}
\label{eq:formula-DelPot-Phi-2}
\Var \left[{\sf \Delta}_a\Pi_\La(z)\right]
=\frac{1}{2\pi}\int_{\C}K(|s|)\Phi(a/s) \, {\diff}m(s)
\end{equation}
where $\Phi$ is defined by~\eqref{eq:def_of_Phi} and
\[
K(s)=-4\pi^2\int_r^\infty\log\Big(\frac{|s|}{t}\Big)k_\La(t)\, t \, {\diff}t.
\]	
We first note that
\begin{align*}
\frac{1}{2\pi}\int_0^{2\pi}\Phi
\Big(\frac{a}{re^{{\rm i}\theta}}\Big){\diff}\theta
&=2\log r - \frac{1}{2\pi}\int_0^{2\pi}
\Big(\log|a-re^{{\rm i}\theta}|
-\log|a+re^{{\rm i}\theta}|\Big){\diff}\theta\\
&=-2\log_+\Big(\frac{a}{r}\Big).
\end{align*}
Incorporating this formula into \eqref{eq:formula-DelPot-Phi-2}, we get
\begin{align*}
\Var \left[{\sf \Delta}_a\Pi_\La(0)\right]&=\int_0^\infty K(r)
\left[\frac{1}{2\pi}\int_0^{2\pi}\Phi
\Big(\frac{a}{re^{{\rm i}\theta}}\Big){\diff}\theta\right]r \, {\diff}r\\
&=-2\int_0^\infty K(r)\log_+\Big(\frac{a}{r}\Big)r \, {\diff}r\\
&=8\pi^2\int_0^\infty\int_0^\infty
\left[\int_{r}^\infty \log\Big(\frac{t}{r}\Big)k_\La(t)t \, {\diff}t\right]
\log_+\Big(\frac{a}{r}\Big)r \, {\diff}r\\
&=8\pi^2\int_0^\infty k_\La(t)t
\left[\int_0^t\log\Big(\frac{t}{r}\Big)
\log_{+}\Big(\frac{a}{r}\Big)r \, {\diff}r\right]{\diff}t.
\end{align*}
The inner integral on the right-hand side simplifies to
\[
\int_0^t\log\Big(\frac{t}{r}\Big)\log_+\Big(\frac{a}{r}\Big)r \, {\diff}r
=t^2\int_0^1 \log\Big(\frac{1}{u}\Big)\log_+\Big(\frac{a}{tu}\Big)u \, {\diff}u,
\]
whence
\begin{align*}
\Var \left[{\sf \Delta}_a\Pi_\La(0)\right]&=
8\pi^2\int_0^\infty k_\La(t)t^3
\left[\int_0^1 \log\Big(\frac{1}{u}\Big)
\log_{+}\Big(\frac{a}{tu}\Big)u \, {\diff}u\right]{\diff}t\\
&=8\pi^2\left(\int_0^a +\int_a^\infty\right)
k_\La(t)t^3
\left[\int_0^1 \log\Big(\frac{1}{u}\Big)
\log_{+}\Big(\frac{a}{tu}\Big)u \, {\diff}u\right]{\diff}t\\
&
\ed 8\pi^2\big(J_1(a)+J_2(a)\big) \, .
\end{align*}
By our assumption $\displaystyle{\int_0^\infty |k_\La(t)|\,t^3 {\diff}t<\infty}$,
\begin{align*}
J_1&=\int_0^a k_\La(t)t^3
\left[\int_0^1 \log\Big(\frac{1}{u}\Big)
\left(\log\Big(\frac{a}{t}\Big)+\log\Big(\frac{1}{u}\Big)\right)
u \, {\diff}u\right]{\diff}t\\
&=C_1\int_0^\infty k_\La(t)t^3\log\Big(\frac{a}{t}\Big){\diff}t+C_2+o(1)
\end{align*}
as $a\to\infty$, where $\displaystyle{C_1=\int_0^1\log(1/u)u \, {\diff}u=\tfrac14}$ and 
\[
C_2=\left(\int_0^\infty k_\La(t)t^3\, {\diff}t\right)\cdot
\left(\int_0^1\log^2\Big(\frac1u\Big)u \, {\diff}u\right).
\]
Moreover, we note that 
\[
\int_0^\infty k_\La(t) t^3\log\Big(\frac{a}t\Big){\diff}t
=\log a\left[\int_0^\infty k_\La(t) \, t^3 \, \frac{\log a-\log t}{\log a}
\done_{[0,a]}(t) \, {\diff}t\right],
\]
and the expression in brackets converges to $\displaystyle{\int_0^\infty k_\La(t)t^3\, {\diff}t}$ 
as $a\to\infty$ by the dominated convergence theorem. Hence, 
the integral $J_1(a)$ satisfies
\begin{equation}
\label{eq:J1-asymp}
J_1(a) =\Big(\frac14\int_0^\infty k_\La(t)t^3 \, {\diff}t\Big)\log a + O(1)
\end{equation}
as $a\to\infty$. For the second integral $J_2(a)$, we have
\begin{align*}
J_2(a)&=\int_a^\infty k_\La(t) \, t^3
\left[\int_0^{a/t}\log\Big(\frac{1}{u}\Big)\cdot 
\log\Big(\frac{a}{tu}\Big)u \, {\diff}u\right]{\diff}t\\
&=\int_a^\infty k_\La(t)t
\left[a^2\int_0^{1}\log\Big(\frac{t}{av}\Big)\cdot 
\log\Big(\frac{1}{v}\Big)v \, {\diff}v\right]{\diff}t\\
&=a^2\int_a^\infty k_\La(t)t\left[C_3\log\Big(\frac{t}{a}\Big)
+C_4\right]{\diff}t \, ,
\end{align*}
for some explicit constants $C_3$ and $C_4$. Furthermore, we have the bound
\begin{align*}
a^2\int_a^\infty k_\La(t) \, t \log\Big(\frac{a}{t}\Big) \, {\diff}t 
&\le a^2\int_a^\infty k_\La(t) \, t\, \log t \, {\diff}t\\
&=a^2\int_a^\infty k_\La(t) \, \frac{t^3}{t^2\log t} \, {\diff}t
\lesssim \log a\int_a^\infty k_\La(t) \, t^3\, {\diff}t = o\big(\log a\big)
\end{align*}
as $a\to\infty$. Combining the asymptotics for 
$J_1(a)$ and $J_2(a)$ with the identity 
$\Var \left[{\sf \Delta}_a\Pi_\La(0)\right]=8\pi^2\big(J_1(a)+J_2(a)\big)$, 
we find that
\[
\Var \left[{\sf \Delta}_a\Pi_\La(0)\right]
=\left(2\pi^2\int_0^\infty k_\La(t)t^3 \, {\diff}t+o(1)\right)\log a
\]
as $a\to\infty$, which is what we wanted.
\end{proof}

\begin{proof}[Proof of Theorem~\ref{thm:flux}]
Assume without loss of generality that
$\Gamma$ starts at the origin and ends at $a>0$, 
and let $\gamma:I\to\Gamma$ denote a parametrization of 
$\Gamma$. Notice first that $\P$-a.s., no point
of $\Lambda$ lies exactly on $R\Gamma$. Hence, 
for each $\lambda\in\Lambda$, we may choose a branch 
of the logarithm such that $\log(\gamma(t)-\la)^{-1}$ 
is continuous on $I$, and we find that
\[
\Re\int_{R\Gamma}\frac{{\diff z}}{z-\lambda}
=\log|Ra-\lambda|-\log|\lambda|.
\]
As a consequence, we find that (see \eqref{eq:Delta-pot-a} for the definition
of ${\sf \Delta} \Pi_a(0)$)
\[
\work_\Lambda(R\Gamma)=\lim_{S\to\infty}\sum_{|\lambda|\le S}
\left(\log|Ra-\lambda|-\log|\lambda|\right)-\tfrac12 c_\La R^2
= {\sf \Delta}_{Ra} \Pi_\La(0),
\]
where the convergence is in the sense of $L^2(\Omega,\bP)$
(this was justified in \S6, Part {\rm I}).
Hence, by applying Theorem~\ref{thm:var-pot}, we find that
\begin{align*}
\Var\left[\work_\La(R\Gamma)\right]=\Var\left[{\sf \Delta}_{Ra} \Pi_\La(0)\right]
&=\left(D_\La+o(1)\right)\log (Ra)\\
&=\left(D_\La+o(1)\right)\log R.
\end{align*}
This completes the proof.
\end{proof}

\section{Rectifiable Jordan arcs with large variance}
\label{s:counter-ex}

\subsection{Nested disks with large charge fluctuations}
\label{s:Gin-counter}
In this section, we specialize to the case when 
$\Lambda$ is the infinite Ginibre point process. 
Denote by $n_\Lambda$ the corresponding point 
count measure. We will show that for any fixed $\epsilon>0$ 
there exists a rectifiable curve $\mathcal{C}_\epsilon$
such that
\begin{equation}
\label{eq:desired_variance_lower_bound}
\Var  
\left[\frac{1}{2\pi {\rm i}}\int_{R\mathcal{C}_\epsilon}
V_\Lambda(z)\,{\diff z}\right] 
\gtrsim R^{2-\epsilon}.
\end{equation}
Let $\ell_k = k^{-1-\epsilon}$. To describe the 
desired curve $\mathcal{C}_\epsilon$, we begin with a 
concatenation of circles (all with positive orientation) 
$\{|z| = \ell_k\}$, and add arcs $\{{\rm i} zt \mid  \ell_k+1 \le t \le \ell_k \}$
which connect subsequent circles along the imaginary axis.
The curve $\mathcal{C}_\epsilon$ is not simple nor closed 
(in the upcoming
section, we will deform $\mathcal{C}_\epsilon$ into a simple Jordan arc), 
but since $(\ell_k)$ is a summable sequence it is rectifiable. 
By the argument principle,
\begin{equation*}
\frac{1}{2\pi {\rm i}} \int_{R\mathcal{C}_\epsilon}
V_\Lambda(z)\,{\diff z} = \sum_{k=1}^{\infty} 
n_\Lambda( R\ell_k \,\D)-\int_{0}^R V_\La({\rm i}w) \, \diff w\,.
\end{equation*}
Since the variance of the last term on the right-hand side is
of order $O(R)$ (for instance, by Theorem~\ref{thm:main}), 
the lower bound~\eqref{eq:desired_variance_lower_bound} 
will follow once we show that
\begin{equation}
\label{eq:variance_lower_bound_point_count}
\Var  \left[\sum_{k=1}^{\infty} n_\Lambda( R\ell_k \, \D )\right]
\gtrsim R^{2-\epsilon}\,,
\end{equation}
which we will do by applying Kostlan's theorem, a result 
which is specific for radially symmetric determinantal point processes 
such as the infinite Ginibre ensemble. 
\begin{thm}[{\cite[Theorem~4.7.1]{GAFBook}}]
\label{thm:kostlan}
Let $\{\,|\lambda_j| : \lambda_j \in \Lambda\}$ be the set of 
absolute values of the Ginibre process ordered by non-decreasing modulus. 
Then $|\lambda_j|$ are independent with
\[
|\lambda_j|^2 \sim \text{\sf{Gamma}} (j,1)\,.
\]
\end{thm}

Here, $\text{\sf{Gamma}}(j,1)$ denotes the standard
Gamma distribution with density given by $f_j(x)=\frac{1}{\Gamma(j)}x^{j-1}e^{-x}$, $x>0$.
We use Theorem~\ref{thm:kostlan} to prove 
\eqref{eq:variance_lower_bound_point_count}. Indeed,
\begin{equation*}
\sum_{k=1}^{\infty} n_\Lambda( R\ell_k \, \D ) 
= \sum_{k=1}^{\infty} \left(\sum_{j=1}^{\infty}
\done_{\{|\lambda_j|^2 \le (R \ell_k)^2\}}\right)
= \sum_{j=1}^{\infty}\left(\sum_{k=1}^{\infty}
\done_{\{|\lambda_j|^2 \le (R \ell_k)^2 \}}\right)
\end{equation*}
and since $\{|\lambda_j|^2\}_{j\ge 1}$ are independent 
we get the lower bound
\begin{equation*}
\Var  \left[\sum_{k=1}^{\infty} 
n_\Lambda(R\ell_k \, \D)\right]\ge \Var  
\left[\sum_{k=1}^{\infty} \done_{\{Z\le(R \ell_k)^2 \}}\right]
\end{equation*} 
where $Z=|\lambda_1|^2 \sim \text{\sf exp}(1)\ed{\sf Gamma}(1,1)$. In view of 
the above, the lower bound~\eqref{eq:variance_lower_bound_point_count} 
follows from the following simple claim.

\begin{claim}
If $Z\sim \text{\sf exp}(1)$ then
\[
\Var  \left[\sum_{k=1}^{\infty} 
\done_{\{Z \le (R \ell_k)^2 \} }\right] \gtrsim R^{2-\epsilon}.
\]
\end{claim}
\begin{proof}
Denote by 
\[
Y_R = \sum_{k=1}^{\infty} \done_{\{Z \le (R \ell_k)^2 \} }\,.
\]
Since $\P(Z\ge t) = \exp(-t)$, we can bound 
the expectation of $Y_R$ as
\begin{align*}
\E[Y_R] = \sum_{k=1}^{\infty} \P\left(Z \le (R \ell_k)^2 \right)
&= \sum_{k=1}^{\infty}(1-e^{-(R\ell_k)^2}) \\ 
& \leq M + R^2\sum_{k=M}^{\infty}\ell_k^2\le M 
+3R^2 M^{-1-2\epsilon} 
\end{align*}
for all $M\ge 1$. By choosing 
$M=\lfloor R^{1/(1+\epsilon)}\rfloor$ we obtain
\[
\E[Y_R] \le 5 R^{1/(1+\epsilon)}\,.
\] 
Furthermore, we have the inclusion of events 
\begin{align*}
\left\{Y_R \ge 6 R^{1/(1+\epsilon)} \right\} 
&\supset\left\{ Z \le \left(R\ell_{\lfloor 
6R^{1/(1+\epsilon)}\rfloor}\right)^2\right\} 
\supset \left\{ Z \le 6^{-2(1+\epsilon)} \right\} 
\supset \left\{ Z \le 0.01\right\}\,.
\end{align*}
Since $Z\sim \text{\sf exp}(1)$, the latter event has 
strictly positive probability, and combining all together we get 
\begin{align*}
\Var  [Y_R] = \E[(Y_R - &\E[Y_R])^2] \\ 
&\gtrsim R^{2/(1+\epsilon)} \, \P(Y_R \ge 6 
R^{1/(1+\epsilon)})\ge R^{2/(1+\epsilon)} \, 
\P(Z \le 0.01)
\end{align*}
as desired.
\end{proof}

\subsection{Nested disks to spirals}
\label{s:counter}

\begin{figure}[t!]
\centering
\begin{subfigure}[t]{.4\textwidth}
\centering
\includegraphics[width=\linewidth]{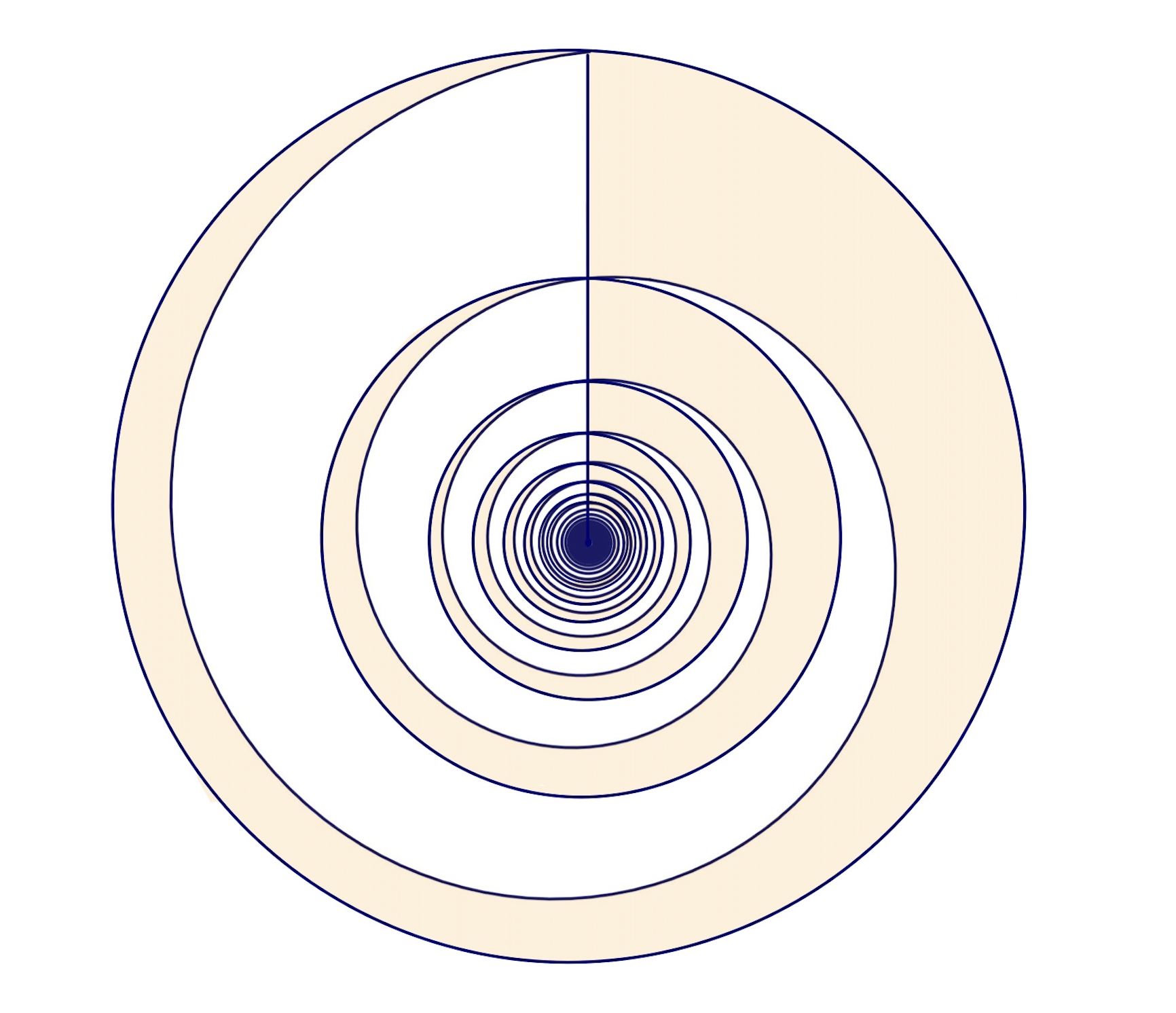}
\end{subfigure}
\hspace{-10pt}
\begin{subfigure}[t]{.4\textwidth}
\includegraphics[width=.835\linewidth]{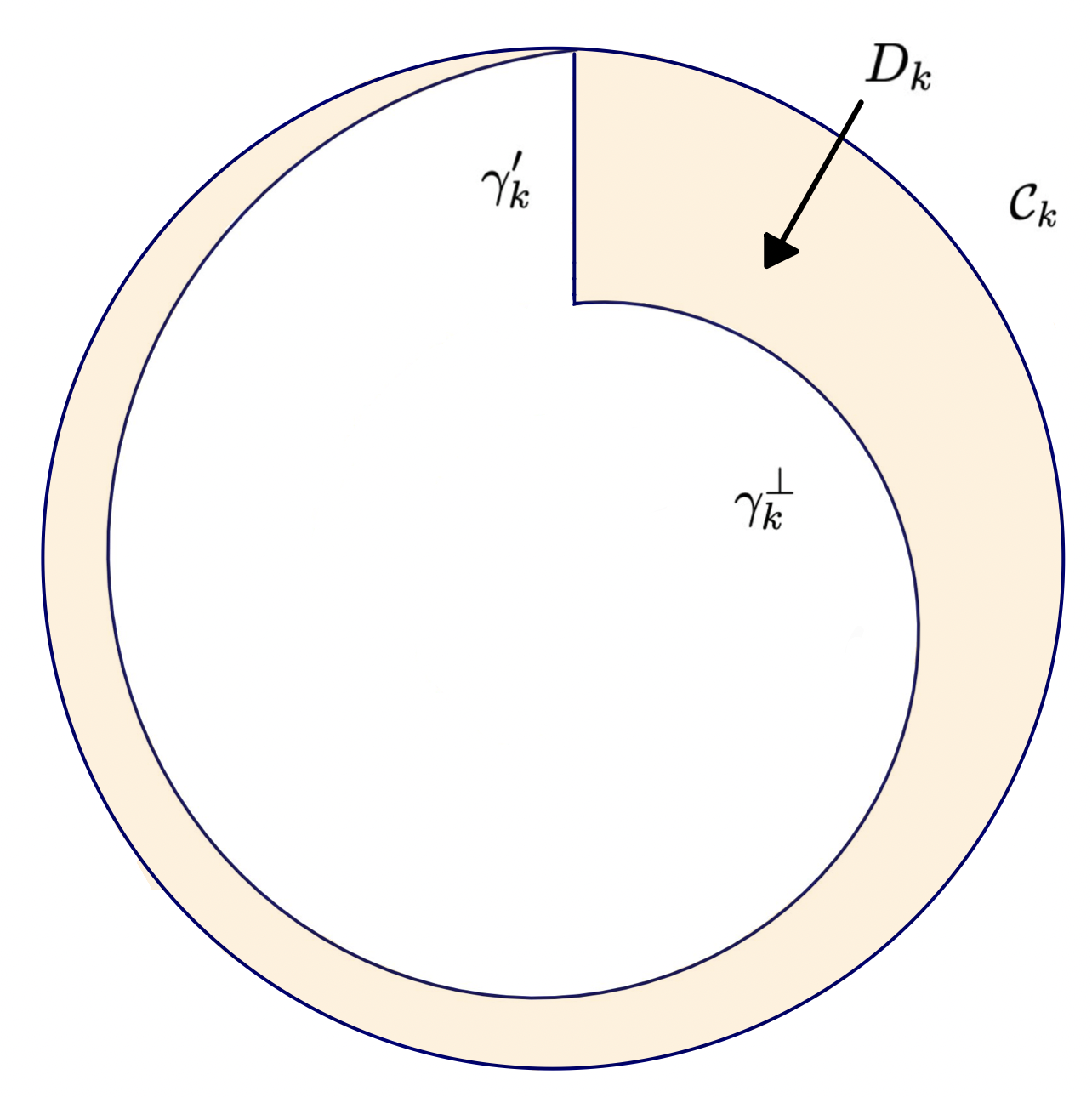}
\end{subfigure}
\caption{{\sf Left:} The sequence of seashell domains (shaded),
each component is enclosed between a circle,
one revolution of the spiral, and a vertical line segment. \\ {\sf Right:}
A single seashell domain $D_k$, with the arcs $\gamma_k^\prime$, $\gamma_k^\perp$
and $\mathcal{C}_k$ indicated.}
\label{fig:seashell}
\end{figure}

We continue our construction of a rectifiable Jordan arc $\Gamma_\epsilon$ with
superlinear growth of $\Var[\act_\La(\Gamma)]$ when $\La$ is the infinite
Ginibre ensemble. For a given $\epsilon>0$, we have
found a curve $\mathcal{C}=\mathcal{C}_\epsilon$ of the form
\[
\mathcal{C}={\rm i}[0,\ell_1]\cup \bigcup_{k\ge 1}\D(0,\ell_k)
\]
such that
\begin{equation}\label{eq:nested-disks-general}
\Var \left[\int_{R\calC}V_\Lambda(z) \, {\diff z}\right]
\gtrsim R^{2-\epsilon}.
\end{equation}
We want to turn $\mathcal{C}$ into a Jordan arc with 
the bound \eqref{eq:nested-disks-general} preserved.
We do so by forming a spiral $\Gamma=\cup_{k\ge 1}\gamma_k$,
where each revolution $\gamma_k$ interpolates
between the intersection of successive circles 
$\calC_k$ and $\calC_{k+1}$ with the imaginary axis; i.e.\
\[
\gamma_k=\left\{\left(\ell_k(1-t) + \ell_{k+1}t\right)e^{2\pi {\rm i} t}:
0\le t\le 1\right\}.
\]
With the help of~\eqref{eq:nested-disks-general}, we prove that
\[
\Var \left[\frac{1}{2\pi {\rm i}}\int_{R\Gamma}V_\Lambda(z) \, {\diff z}\right]
\gtrsim R^{2-\epsilon}
\]
as well. To see why this is so, we first recall that for the 
Ginibre point process, the (radial) two-point function $k_\La(t)$ is given by
\[
k_\La(t) = -\pi^{-2} e^{-\pi t^2}\, ,
\]
see \S2.4, Part {\rm I}. Integration by parts yields
\begin{equation*}
K(z) = 4 \int_{|z|}^{\infty} \log\Big(\frac{r}{|z|}\Big) r e^{-\pi r^2} \, {\diff} r 
= \frac{2}{\pi} \int_{|z|}^{\infty} e^{-\pi r^2} \, \frac{\diff r}{r} \, ,
\end{equation*}
which implies that
\[
\bar \partial K(z) = \frac{e^{-\pi |z|^2}}{\pi \bar z}  \in L^1(\C, \diff m)\, .
\]
We let $\gamma'_k$ denote the
short vertical segment
$\gamma'_k \stackrel{{\rm def}}{=} \{{\rm i} t: \ell_{k+1}\le t\le \ell_k\}$.
Then we obtain a closed curve by taking the union
$\calC_k\cup\gamma_k^\perp\cup\gamma'_k$ (traversing $\calC_k$ with its original orientation, and
where $\gamma_k^\perp$ is $\gamma_k$ traversed backwards),
which encloses a domain $D_k$;
cf.\ Figure~\ref{fig:seashell}. All the
domains $D_k$ are disjoint, and
we denote their union by $D$ (the \say{seashell}).

By Green's formula, we have for each $k\ge 0$ that
\[
\int_{R(\calC_k\cup\gamma_k^\perp\cup\gamma'_k)}K(|x-y|) \, \diff y=2{\rm i}\int_{R D_k}
\bar\partial_y K(|x-y|) \, \diff m(y),
\]
so that 
\begin{align*}
\int_{R\gamma_k}K(|x-y|) \, \diff y &=-\int_{R\gamma_k^\perp}K(|x-y|) \, 
\diff y \\ &= \int_{R(\mathcal{C}_k\cup\gamma_k')}K(|x-y|) \, \diff y
-2{\rm i}\int_{R D_k}\bar\partial_y K(|x-y|) \, \diff m(y).
\end{align*}
Adding up the contributions for $k\ge 0$, we obtain
\begin{align}
\int_{R\Gamma}K(|x-y|) \, \diff y&
=\int_{R(\calC \cup i[0,\ell_1])}K(|x-y|) \, \diff y
- 2{\rm i}\int_{R D}\bar\partial_y K(|x-y|) \, \diff m(y)
\\&=\int_{R\calC}K(|x-y|) \, \diff y + O(1),
\end{align}
where we have used that fact that $\bar\partial K\in L^1(\C,\diff m)$.
Integrating over $R\Gamma$, we find that
\begin{align}
\iint_{R\Gamma\times R\Gamma}K(|x-y|) \, \diff y\diff \bar{x}
& =\int_{R\calC}\left[\int_{R\Gamma}K(|x-y|) \, \diff \bar{x}\right]\diff y+O(R)
\\ &=\iint_{R\Gamma\times R\calC}K(|x-y|) \, \diff y\diff \bar{x}+O(R).
\end{align}
Repeating the above
argument, we finally arrive at
\[
\iint_{R\Gamma\times R\Gamma}K(|x-y|) \, \diff y\diff \bar{x}
=\iint_{R\calC\times R\calC}K(|x-y|) \, \diff y\diff \bar{x}+O(R),
\]
from which the claim follows.

\subsection*{Acknowledgments}

We are grateful to Fedor Nazarov, Alon Nishry and Ron Peled for
several useful discussions throughout our work. 
We thank the anonymous referees for helpful comments.

Part of the work on this project was carried out 
while A.W. was based at Tel Aviv University.
He would like to express his gratitude for the 
excellent scientific environment provided there.

The work of A.W.\ was 
supported by the KAW foundation grant 2017.0398, by 
ERC Advanced Grant 692616 and by Grant No. 2022-03611 
from the Swedish Research Council (VR).
The work of M.S.\ and O.Y.\ was supported by 
ERC Advanced Grant 692616, ISF Grant 1288/21 and
by BSF Grant 202019.

\subsection*{Data availability statement}
Data sharing not applicable to this article as no datasets 
were generated or analysed during the current study.

\subsection*{Compliance with ethical standards}
The authors do not have any potential conflicts of interests to disclose.

\bigskip
\bigskip
\bigskip

\noindent Sodin:
School of Mathematical Sciences, Tel Aviv University, Tel Aviv, Israel
\newline {\tt sodin@tauex.tau.ac.il}
\smallskip\newline\noindent{Wennman: Department of Mathematics, 
KTH Royal Institute of Technology, Stockholm, Sweden
\newline {\tt aronw@kth.se}
\smallskip\newline\noindent Yakir:
School of Mathematical Sciences, Tel Aviv University, Tel Aviv, Israel
\newline {\tt oren.yakir@gmail.com}
}
\end{document}